\documentclass{article}
\usepackage[utf8x]{inputenc}
\usepackage[OT1]{fontenc}
\usepackage{graphicx}

\usepackage{amsmath, amsthm, latexsym, amsfonts, amssymb, amscd,bbm}
\usepackage[english]{babel}

\def \R{\mathbb R}

\newtheorem{thm}{Theorem}[section]
\newtheorem{rem}{Remark}
\newtheorem{theorem}[thm]{Theorem}
\newtheorem{lemma}[thm]{Lemma}
\newtheorem{proposition}[thm]{Proposition}
\newtheorem{assump}{Assumption}

\newcommand{\norm}[1]{\left \| #1 \right \|}

\renewcommand{\tilde}{\widetilde}

\newcommand{\balpha}{\boldsymbol\alpha}
\newcommand{\bbeta}{\boldsymbol\beta}

\newcommand{\bupsilon}{\boldsymbol\upsilon}
\newcommand{\bgamma}{\boldsymbol\gamma}

\numberwithin{equation}{section}

\title{Metastability of Waves, Patterns and Oscillations Subject to Spatially-Extended Noise}

\author{J. MacLaurin\footnotemark[1]}
\begin{document}
\maketitle
\newcommand{\slugmaster}{%
\slugger{siads}{xxxx}{xx}{x}{x--x}}

\begin{abstract}
In this paper we present a general framework in which one can rigorously study the effect of spatio-temporal noise on traveling waves, stationary patterns and oscillations that are invariant under the action of a finite-dimensional set of continuous isometries (such as translation or rotation). This formalism can accommodate patterns, waves and oscillations in reaction-diffusion systems and neural field equations. To do this, we define the phase by precisely projecting the infinite-dimensional system onto the manifold of isometries. Two differing types of stochastic phase dynamics are defined: (i) a variational phase, obtained by insisting that the difference between the projection and the original solution is orthogonal to the non-decaying eigenmodes, and (ii) an isochronal phase, defined as the limiting point on manifold obtained by taking $t\to\infty$ in the absence of noise. We outline precise stochastic differential equations for both types of phase. The variational phase SDE is then used to show that the probability of the system leaving the attracting basin of the manifold after an exponentially long period of time (in $\epsilon^{-2}$, the magnitude of the noise) is exponentially unlikely. In the case that the manifold is periodic (such as for spiral waves, spatially-distributed oscillations, or neural-field phenomena on a compact domain), the isochronal phase SDE is used to determine asymptotic limits for the average occupation times of the phase as it wanders in the basin of attraction of the manifold over very long times. In particular, we find that frequently the correlation structure of the noise biases the wandering in a particular direction, such that the noise induces a slow oscillation that would not be present in the absence of noise.
\end{abstract}
\renewcommand{\thefootnote}{\fnsymbol{footnote}}
\footnotetext[1]{New Jersey Institute of Technology. james.n.maclaurin@njit.edu}
\section{Introduction}
Spatially-extended patterns and waves are ubiquitous in the biological and physical sciences and are a key lens through which emergent phenomena are understood. Furthermore biology is typically very noisy, and thus it is of great importance to understand the effect of stochasticity on these patterns and waves \cite{Panja2004,Lindner2004,Sagues2007}. The literature on stochastic patterns and waves includes general Turing patterns \cite{Biancalani2017}, the Allen-Cahn / Cahn-Hilliard equation \cite{Grafke2017}, waves and patterns in the stochastic Brusselator \cite{Biancalani2010,Biancalani2011}, patterns in neural fields \cite{Hutt2008,Fung2010,Kilpatrick2013a,Gong2013,Thul2016,Avitable2017,Kuehn2019,Bressloff2019e,MacLaurin2020a}, interfaces in the Ginzburg-Landau equation \cite{Brassesco1995,Harris2006}, the stochastic burger's equation \cite{Blomker2007} and the effect of spatially-distributed noise on traveling waves \cite{Mikhailov1983,Armero1998,Cartwright2019}, such as the FKPP traveling waves \cite{Doering2003,Brunet2006}, invasion waves in ecology \cite{Lewis2000}, the stochastic Nagumo equation \cite{Kruger2017,Hamster2020,Geldhauser2020}, geometric waves \cite{Zdzislaw2019} and numerical methods for stochastic traveling waves \cite{Lord2012}. Good reviews of the literature on the effect of noise on traveling waves can be found in \cite{Panja2004,Sagues2007,Kuehn2019a}.

Consider the deterministic unforced system with solution $u_t \in \mathcal{C}_b(\mathbb{R}^d , \mathbb{R}^N)$ (the Banach space of continuous bounded $\R^N$-valued functions over $\R^d$, for $N,d\geq1$)
\begin{equation}\label{eq: deterministic dynamics}
\frac{d u_t(x)}{d t}= A u_t(x) + f(u_t)(x)
\end{equation}
Here $A$ is a linear operator, such as the Laplacian, and $f$ is a continuous function that is locally bounded. Frequently the fixed points of such systems possess several symmetries (typically invariance under translation and / or rotation), and there exists a smooth attracting manifold of fixed points
(parameterized by $(\varphi_{\balpha})_{\balpha\in\R^m} \subset  \mathcal{C}_b(\mathbb{R}^d , \mathbb{R}^N))$ such that
\begin{equation}
\label{stationary equation}
A\varphi_{\balpha}+ f(\varphi_{\balpha}) = 0,\ \forall \balpha\in\R^m.
\end{equation}
The thrust of this article is to drive \eqref{eq: deterministic dynamics} by space-time white noise and understand how the noise shapes the dynamics in the basin of attraction of the manifold $(\varphi_{\balpha})_{\balpha\in\R^m} $. As such, we consider the following stochastic evolution equation driven by space-time noise $W_t$ \footnote{A cylindrical $H$-valued Wiener process.}. The stochastic evolution equation takes values in the Banach space $E_0 := \varphi_{0} + H$, where $H$ is a Hilbert Space, and has the form
\begin{equation}
\label{main}
du_t = [Au_t + f(u_t)]dt + \varepsilon B(t , u_t)dW_t,
\end{equation}
for some $\varepsilon>0$.   In \cite{inglis2016general} we developed a method of precisely projecting the stochastic dynamics onto the manifold parameterized by $\lbrace \varphi_{\balpha} \rbrace_{\balpha \in \mathbb{R}^m}$, generalizing deterministic work on the orbital stability of a manifold of fixed points (see  \cite[Chapter 5]{Volpert1994}  and \cite[Chapter 4]{Kapitula2013}) to a stochastic setting. This paper builds on the work of \cite{inglis2016general} and has three main aims: (i) to obtain an SDE yielding a precise projected phase dynamics, (ii) obtain accurate probabilistic exit-time estimates over long periods of time and (iii) to obtain long-time `occupation-time' estimates for the wandering of the stochastic phase on the manifold $\mathbb{R}^m$. In more detail:
\begin{itemize}
\item In \cite{inglis2016general} the manifold was taken to be one-dimensional; in this paper the manifold is multidimensional.
\item In \cite{inglis2016general} the stability estimates in Section 6 require the linearization of \eqref{stationary equation} to be self-adjoint and immediately contractive. This assumption is relaxed, and now the main requirement is that it possesses a spectral gap (except for the neutral eigenmodes tangential to the manifold in \eqref{stationary equation}).
\item In \cite{inglis2016general} a `variational phase' SDE was obtained. In this work an SDE for the variational phase is also obtained, and an additional `isochronal phase' SDE is also obtained.
\item In \cite{inglis2016general} the noise was additive and a `Q-Wiener process' \cite{DaPrato2014}, meaning in effect that there is sufficient spatial correlation that the space-time noise is spatially-continuous. In this paper we consider multiplicative noise, and when the operator $A$ is sufficiently smoothing (such as when $A$ is the sum of a Laplacian and possibly first order derivatives), the noise can be `cylindrical', meaning that the driving noise is spatially-decorrelated and spatially-discontinuous. 
\item The bound on the growth of the error in \cite[Corollary 6.3]{inglis2016general} is suboptimal. This bound was greatly improved in \cite{MacLaurin2020a} for `bumps' of activity in stochastic neural fields. In this work we use a similar method to \cite{MacLaurin2020a} to show that the probability of the system leaving the manifold of translated bump solutions over an exponentially long period of time (i.e. $T \simeq \exp( C\epsilon^{-2})$) is exponentially unlikely. In other words, we are in the Large Deviations regime \cite{Salins2019}.
\end{itemize}

After one has shown that, with very high probability, the system spends a very long period of time in the neighborhood of the manifold, it is natural to investigate the `occupation time': that is, the typical proportion of time that it spends in particular neighborhoods of the manifold over long time intervals (see \cite{Donsker1975} for a classical result on the occupation time of Markov chains). Doing this yields an understanding of how the correlation structure of the noise interacts with the geometry of manifold $\lbrace \varphi_{\balpha}\rbrace_{\balpha \in \mathbb{R}^m}$ to shape the wandering of the stochastic phase over long periods of time. Under the assumption that the manifold is periodic, we find that the natural timescale for the induced phase dynamics on the manifold is $\epsilon^{-2}t$ ($\epsilon$ being the magnitude of the white noise), and we are then able to demonstrate that the occupation time converges to that of the invariant measure of the rescaled process (modulo the periodicity). In many cases this result implies that the noise correlation structure can induce a slow \textit{oscillation} in the phase over long periods of time. To demonstrate the convergence of the occupation time, we must employ a slightly different phase that is analogous to the isochronal phase definition for stochastic oscillators \cite{Ermentrout2010}. It is worth comparing the results of this paper to those of Blomker \cite{Blomker2007a}. He uses a rescaling of time to understand the projection of fluctuations in SPDEs onto a manifold that is weakly unstable (by contrast, this paper concerns a projection onto a stable manifold of fixed points).

To the best of this author's knowledge, the first rigorous work on the behavior of stochastic systems near an attracting manifold is that of Katzenberger \cite{Katzenberger1991}. He determined limiting equations for a finite-dimensional jump Markov process pulled onto a manifold by a large drift. This theory has been applied in finite-dimensional stochastic models of population dynamics, including in \cite{Constable2013,Parsons2017,Grafke2017,Popovic2020}. In the deterministic literature, there is a well developed literature on the orbital stability of finite-dimensional submanifolds in infinite-dimensional spaces, including work on the Ginzburg-Landau equation \cite{Beyn2004}, traveling waves \cite{Kapitula2013,Volpert1994}, and spiral waves \cite{Sandstede1997}: see \cite{Volpert1994,Kapitula2013} for many more examples.

It is well-established \cite{Armero1998,Panja2004,Gottwald2004,Brunet2006,Sagues2007,Bressloff2012} that for traveling waves perturbed by space-time noise, one can identify the leading order diffusive flux of the wave position $\beta_t$ over timescales of $O(\epsilon^{-2})$ by matching the leading order terms in the equation 
\begin{equation}
\langle u_t - \varphi_{\beta_t} , \psi_{\beta_t} \rangle \simeq 0.\label{eq: approximate wave position}
\end{equation}
 Here $\psi_{\beta_t}$ is the neutral eigenvector of the adjoint of the linearization about the traveling wave positioned at $\beta_t$. This is directly motivated by deterministic theory (see Chapter 4 of \cite{Kapitula2013} and Chapter 5 of \cite{Volpert1994}). In stochastic systems, this technique has been employed in reaction-diffusion traveling fronts by \cite{Armero1998,Panja2004,Brunet2006,Sagues2007,inglis2016general,Cartwright2019,Hamster2020,Hamster2020a,Hamster2020b}, in stochastic neural fields by Bressloff and co-workers \cite{Bressloff2012,inglis2016general,MacLaurin2020a} and Kilpatrick and Ermentrout \cite{Kilpatrick2013a}. Hamster and Hupkes  \cite{Hamster2020a} demonstrate that the probability of leaving a neighborhood of the manifold over an exponentially long period of time goes to zero as the noise strength goes to zero. The work of Cartwright and Gottwald \cite{Cartwright2019} is interesting because the manifold they project onto also includes a non-neutral eigenmode, and there is therefore potential for their ansatz to be accurate over longer periods of time.
 
Despite these strengths, none of the works in the previous paragraph (with the exception of \cite{MacLaurin2020a}) determine precise expressions for the quadratic variation (i.e. $d\beta_t d\beta_t$) and cross-variation (i.e. $d\beta_t du_t$) terms that would be necessary for \eqref{eq: approximate wave position} to be satisfied exactly, and this creates problems once one wishes to precisely understand the phase dynamics over longer timescales (that diverge on timescales greater than $O(\epsilon^{-2})$). Indeed in the deterministic setting, one knows from the implicit function theorem that \eqref{eq: approximate wave position} can be solved for $\beta_t$ in some neighborhood only if $\partial / \partial \beta_t \lbrace \langle u_t - \varphi_{\beta_t} , \psi_{\beta_t} \rangle \rbrace \neq 0$, and if $\partial / \partial \beta_t \lbrace \langle u_t - \varphi_{\beta_t} , \psi_{\beta_t} \rangle \rbrace$ asymptotes to zero as $t$ approaches some limit, then the coefficients in the ODE for $\beta_t$ will blow up. It was demonstrated in \cite{inglis2016general} that this blowup can also occur in the stochastic setting: in the SDE for the phase $d\beta_t$, both the drift and the diffusion coefficient blowup if $\partial / \partial \beta_t \lbrace \langle u_t - \varphi_{\beta_t} , \psi_{\beta_t} \rangle \rbrace$ approaches zero. One cannot \textit{a priori} rule out $\partial / \partial \beta_t \lbrace \langle u_t - \varphi_{\beta_t} , \psi_{\beta_t} \rangle \rbrace$ asymptoting to zero in the stochastic setting (by definition, noisy systems exhibit a diversity of behavior, with various degrees of probability); instead one must try to derive accurate bounds on the probability of this not occurring. Indeed recent work by this author and Bressloff \cite{MacLaurin2020a} has solved \eqref{eq: approximate wave position} exactly for a neural field equation, and used this to prove that the probability of the system leaving a neighborhood of the manifold of bump solutions after an exponentially long period of time (i.e. $T \simeq \exp(C \epsilon^{-2})$) is exponentially small (i.e. $\text{Prob} \simeq \exp(-C\epsilon^{-2})$). In other words, instead of just determining the statistics for the leading order diffusion of the wavefront, MacLaurin and Bressloff determine an exact nonlinear SDE, coupled in the same space as the driving noise, and which is accurate over very long timescales. In summary, for a rigorous and accurate long-time expression for the phase SDE, one desires (i) an exact solution to \eqref{eq: approximate wave position} (we also determine an `isochronal phase' in Section , that agrees with \eqref{eq: approximate wave position} to leading order), and (ii) control over $\partial / \partial \beta_t \lbrace \langle u_t - \varphi_{\beta_t} , \psi_{\beta_t} \rangle \rbrace$ going to zero (this is encapsulated in the stopping time $\tau$) and (iii) strong bounds on the fluctuations of the component of the solution orthogonal to the manifold.
 
Prior to the work of this author in \cite{inglis2016general}, Stannat and co-workers in \cite{stannat2014,Kruger2017} (for reaction-diffusion systems) and \cite{kruger2014} (for neural field systems) approximated the dynamics of the phase $\beta_t$ by an ordinary differential equation with $\frac{d\beta_t}{dt}$ proportional to $\frac{\partial}{\partial\beta_t}\|u_t - \varphi_{\beta_t}\|^2$. They then decomposed $u_t - \varphi_{\beta_t}$ into an Ornstein-Uhlenbeck Process plus a remainder term of lower order in $\epsilon$. They rigorously proved that this ansatz is accurate in the small $\epsilon$ limit.

There exists a literature on the first-exit-time from attracting wells in infinite-dimensional stochastic systems \cite{Berglund2013,Barret2015}, generalizing the classical Kramer's Law (see \cite{Dembo1998}) to an infinite dimensional setting. Usually one obtains these estimates through proving a Large Deviations principle for the system \cite{Freidlin2012,Sowers1992,Cerrai2004,Kuehn2014,Salins2019,Zdzislaw2019}. A helpful explanation of how a Large Deviations Principles usually implies that a system spends an exponentially long period of time in the neighbourhood of the attracting fixed point can be found in \cite[Section 5.6]{Dembo1998}: these estimates were originally derived by Freidlin and Wentzell \cite{Freidlin2012}. A recent preprint of Salins and Spiliopoulos \cite{Salins2019a} has determined the exponential asymptotics of the first exit time from the attracting basin of a fixed point in an SPDE. By contrast, the asymptotics in this paper study the first exit time from the attracting basin of a smooth manifold of fixed points.

The organization of the paper is as follows. In Section \ref{General Setting} we describe the general setting we consider, and outline the necessary assumptions.   Section \ref{examples} then goes on to list some examples that fit into the general setting, including reaction-diffusion traveling waves, spiral waves, spatially-distributed oscillations and patterns in neural fields.  In Section 4 we define two different phase equations: the variational phase, obtained by defining the position of the pattern / wave to be such that the orthogonal amplitude is precisely perpendicular to the neutral eigenvectors of the adjoint operator, and the isochronal phase, obtained by removing noise from the system and taking $t\to\infty$. In Section \ref{Section Long Time Stability} we demonstrate that the probability of the system leaving a neighborhood of the manifold before an exponentially long period of time is exponentially unlikely. Finally in Section \ref{Section Ergodicity}, assuming the manifold to be periodic, we determine the limiting probability distribution for the phase (modulo the periodicity), and the average shift induced in the phase by the noise over long periods of time.

\vspace{0.3cm}
\noindent\textit{Notation:}
As usual, $\mathcal{C}(\R^d)$ and $\mathcal{C}^\infty(\R^d)$ will denote the spaces of real-valued functions on $\R^d$ that are continuous and smooth respectively.  Moreover $L^p(\R^d)$ ($p\geq1$), will be the space of $p$-integrable functions with respect to the Lebesgue measure on $\R^d$.  Finally, for general Banach spaces $E_1, E_2$, we will denote by $\mathcal{L}(E_1, E_2)$ the space of bounded linear operators $:E_1 \to E_2$. $\mathcal{L}(E_1,E_2)$ is equipped with the operator norm.

Let $H:=[L^2(\R^d)]^N$, equipped with the standard inner product denoted by $\langle\cdot, \cdot \rangle$ and norm $\|\cdot\|$. Let $\mathcal{L}_{HS}$ be the space of all linear Hilbert-Schmidt operators $H \to H$, with the norm of $B \in \mathcal{L}_{HS}$ written as
\begin{equation}\label{eq: HS norm}
\norm{B}_{HS} = \sum_{j=1}^\infty \langle e_j , B e_j \rangle,
\end{equation}
where $\lbrace e_j \rbrace_{j\in\mathbb{Z}^+}$ is any orthonormal basis for $H$. The operator norm of any $U \in \mathcal{L}(H,H)$ is written as
\begin{equation}
\norm{U}_{\mathcal{L}} = \sup\big\lbrace \norm{Uz} \; : \norm{z} = 1 \big\rbrace.
\end{equation}
Let $E:= \varphi_0 + H_{0}$ (i.e. $u\in E$ if and only if $u = \varphi_0 + v$ for some $v$ in $H_0$), endowed with the topology inherited from $H_0$. For $y,z \in E$, with $y = \varphi_0 + y_0$ and $z = \varphi_0 + z_0$, we write $\norm{y-z}_0 := \norm{y_0 - z_0}_0$.

\section{Problem Setup}
\label{General Setting}

\label{sec:stoch wave}

We consider the system in \eqref{eq: deterministic dynamics}, forced by multiplicative spatially-extended white noise. This yields an $E$-valued stochastic differential equation ($E = \varphi_0 + H$), where 
\begin{equation}
\label{main}
du_t = [Au_t + f(u_t)]dt + \varepsilon B(t , u_t)dW_t,
\end{equation}
for some $\varepsilon>0$. For each $t\geq 0$ and $u_t \in E$, $B(t,u_t)$ is a linear operator on $H$: it specifies the multiplicative effect of the noise. We take the initial condition to be $u_0 := \varphi_{\bar{\bbeta}}$ for some constant $\bar{\bbeta} \in \mathbb{R}^m$ (one could easily obtain analogous results by taking $u_0$ to be in a small neighborhood about  $ \varphi_{\bar{\bbeta}}$). $A$ is a linear operator (such as the Laplacian, for reaction-diffusion systems) and $f$ is a Lipschitz nonlinear function on $E$. As noted in the introduction, we assume that $\lbrace \varphi_{\balpha} \rbrace_{\balpha \in \mathbb{R}^m}$ constitute a manifold of fixed points, such that
\begin{equation}
\label{stationary equation}
A\varphi_{\balpha}+ f(\varphi_{\balpha}) = 0,\ \forall \balpha\in\R^m.
\end{equation}
$W_t$ is `space-time white noise', such that formally, for any $v,y \in H$,
\begin{equation}\label{eq: white noise definition}
\mathbb{E}\big[ \langle v , W_t \rangle \big] = 0 \; \; , \; \; \mathbb{E}\big[ \langle v , W_t \rangle \langle y, W_t \rangle \big] = t\langle v,y \rangle.
\end{equation}
Some care is required to make precise sense of (i) the definition of the white noise in \eqref{eq: white noise definition} and (ii) what one means by a solution to \eqref{main}. If the white noise were to be completely spatially decorrelated, then one cannot take a spatial derivative, and one therefore cannot use standard partial differential equation theory to make sense of a solution to \eqref{main}. Indeed one cannot even properly define an $H$-valued Gaussian random variable satisfying the properties in \eqref{eq: white noise definition}. However if $A$ is the Laplacian, then one knows that it works to smooth functions. Indeed since $\frac{d^2}{dx^2}\cos( a x) = -a^2\cos(ax)$ and $\frac{d^2}{dx^2}\sin( a x) = -a^2\sin(ax)$, one knows that the highly fluctuating components of a solution get strongly damped by the Laplacian. Thus the solution that we are going to define can be thought of as taking the limit of \eqref{main} for increasingly decorrelated noise, but such that a sensible limit is obtained because high wavenumber fluctuations get damped down by the smoothing action of $A$. The theory of stochastic partial differential equations has been developed to make precise sense of this limit \cite{DaPrato2014,Liu2015}. 

Formally, we define $W_t$ to be a cylindrical $H$-valued Wiener process on the filtered probability space $(\Omega, \mathcal{F}, \{\mathcal{F}_t\}_{t\geq0}, \mathbb{P})$ . The solution described in the previous paragraph can be precisely defined by employing a stochastic analog of the variation-of-constants solution in PDEs, as stated in the following proposition.
\begin{proposition}
\label{K-S}
Assume the assumptions of Section \ref{eq: technical}. Then stochastic evolution equation \eqref{main} has a unique mild solution, which can be decomposed (in a non-unique way) as $u_t = \varphi_{\bar{\bbeta}} + v^{\bar{\bbeta}}_t$ where $(v^{\bar{\bbeta}}_t)_{t\geq0}$ is the unique weak (and mild) $H$-valued solution to
\[
dv^{\bar{\bbeta}}_t  = [Av^{\bar{\bbeta}}_t+ f(\varphi_{\bar{\bbeta}} + v^{\bar{\bbeta}}_t) -  f(\varphi_{\bar{\bbeta}})]dt + \varepsilon B(t , u_t) dW_t, \quad t\geq0,
\]
with initial condition $v^{\bar{\bbeta}} = 0$ i.e.
\[
v^{\bar{\bbeta}}_t  =  \int_0^tP^A_{t-s}\left[f(\varphi_{\bar{\bbeta}} + v^{\bar{\bbeta}}_s) -  f(\varphi_{\bar{\bbeta}})\right]ds + \varepsilon\int_0^t P^A_{t-s}B(s, u_s) dW_s, \quad t\geq0.
\]
and $(P^A_t)_{t\geq0}$ is the semigroup generated by $A$.
\end{proposition}
\begin{proof}
The proof of this result is a straightforward application of \cite[Theorem 7.4]{DaPrato2014} using the globally Lipschitz assumption on $f$ (Assumption \ref{assump f} (ii)), the fact that $A$ generates a $\mathcal{C}_0$-semigroup on $H$ (Assumption \ref{assump A} (i)) and the assumptions on $B$ above. 
\end{proof}

\begin{rem}\label{Important Remark}
We remark that for traveling waves, \eqref{main} is in the the moving coordinate frame.  To illustrate what we mean by this, suppose again we are in the concrete situation of the standard neural field equation described in Section \ref{sec:traveling fronts}, so that there is a solution $\hat{u}(x-ct)$ to \eqref{NF} for some speed $c$.  The stochastic version of this equation with purely additive noise would then be $du_t = [-u_t + w*F(u_t)]dt + B(u_t)dW_t$.
In the moving frame (i.e. under the change of variable $x\mapsto x-ct$), the equation becomes
\[
du_t = [Au_t+ w*F(u_t)]dt + \tilde{B}(t , u_t )dW_t^Q,
\]
where as above $Au = cu' $, $u\in\mathcal{D}(A)$ and now $\tilde{B}(t , u_t)w := B( \tilde{w} )$, $\tilde{w}(x) = w(x+ct)$ for $w\in E$.
\end{rem}


\subsection{General Assumptions}\label{eq: technical}

Our assumptions on the drift dynamics are intended to resemble the assumptions in the deterministic theory outlined in \cite[Chapter 4]{Kapitula2013} as much as possible. We thus assume that the manifold of fixed points of \eqref{stationary equation} can be obtained by applying a smooth isometry $\mathcal{T}_{\balpha}: E \to E$, in the following manner,  
\begin{equation}
\varphi_{\balpha} = \mathcal{T}_{\balpha}\varphi_{0}.
\end{equation}
We employ the following assumptions on the family of isometries $\lbrace \mathcal{T}_{\balpha} \rbrace$.
\begin{assump}
For $\balpha,\bbeta \in \mathbb{R}^m$,
\begin{align}
 \mathcal{T}_{\balpha} \circ \mathcal{T}_{\bbeta} &= \mathcal{T}_{\balpha + \bbeta} \\
 A\cdot \mathcal{T}_{\balpha} &= \mathcal{T}_{\balpha}\cdot A \; \text{ and }\;
 f\big(\mathcal{T}_{\balpha}\cdot u\big) = \mathcal{T}_{\balpha}\cdot f(u).
\end{align}
\end{assump}
The following assumption on the generator is satisfied in the vast majority of interesting applications, including for $A$ being an elliptic operator (for waves and patterns in reaction diffusion systems), or hyperbolic (for traveling / rotating waves in neural field equations, in the co-moving reference frame).
\begin{assump}
\label{assump A}
The domain of $A$, i.e. $\mathcal{D}(A)$, is dense in $H$, and the restriction of $A$ to $H$ (also denoted by $A$) is the generator of a $\mathcal{C}_0$-semigroup $P^A_t$ on $H$.  \newline
\end{assump}
\begin{assump}
\label{assump f}
Assume that the nonlinear function $f$ acting in $E$ is such that:
\begin{itemize}
\item [(i)] $f$ is defined on all of $E$, and for all $u\in E$ there exists the Frechet Derivative for perturbations in $H$, written $Df(u) \in L(H,H)$, i.e.  such that for all $v \in H$,
\[
\lim_{h\to 0}\norm{\frac{f(u + hv) - f(u)}{h} - Df(u)\cdot v} = 0,
\]
and $u \to DF(u)$ is continuous.
\item[(ii)] For all $u\in E$ there exists the second Frechet Derivative $D^{(2)}f(u) \in L(H \times H,H)$ such that for all $v,w \in H$,
\[
\lim_{h\to 0}\norm{\frac{Df(u + h w)\cdot v - Df(u)\cdot v}{h} - D^{(2)}f(u)\cdot v\cdot w} = 0,
\]
and $u \to D^{(2)}f(u)$ is continuous.
\item[(iii)] The third Frechet Derivative $D^{(3)}f(u) \in L(H \times H \times H ,H)$ exists, is such that $u \to D^{(3)}f(u)$ is continuous. 
\item[(iv)] $\sup_{u\in E} \|Df(u)\|_{L(H,H)}<\infty$ (so that $H\ni v \mapsto f(\varphi_{\balpha}  + v)$ is globally Lipschitz $\forall \balpha\in\R^m$);
\item[(v)] 
\begin{align}
\sup_{u\in E, v,w\in H, \norm{v},\norm{w} \leq 1} \|D^{(2)}f(u)\cdot v \cdot w\| &< \infty \\
\sup_{u\in E, v,w,z\in H, \norm{v},\norm{w},\norm{z} \leq 1} \norm{ D^{(3)}f(u)\cdot v \cdot w \cdot z} &< \infty.
\end{align}
\end{itemize}
\end{assump}
In order that we can project the stochastic dynamics onto the manifold parameterized by $\lbrace \varphi_{\balpha} \rbrace_{\balpha \in \mathbb{R}^m}$, we require more smoothness assumptions than in the deterministic case \cite{Kapitula2013}. The reason for this is that realizations of the stochastic process are not differentiable in time (this is why the stochastic analog of the Chain Rule  - Ito's Lemma - requires second derivatives).
\begin{assump}
\label{assump varphi}
Assume that the family $(\varphi_{\balpha})_{\balpha\in\R^m}$ satisfies the following conditions.
\begin{itemize}
\item [(i)] For $1\leq i\leq m$, the derivatives $\varphi_{\balpha ,i} := [\partial /\partial \alpha_i  ] \varphi_{\balpha}$ and $\varphi_{\balpha ,ij} := [\partial^2 /\partial \alpha_i \partial \alpha_j  ] \varphi_{\balpha}$ exist (the derivatives being taken in the norm of the space $H$) and are all in the space $H$.
\item[(ii)] $\varphi_{\balpha}- \varphi_{\bbeta} \in H$ for any $\balpha,\bbeta \in \mathbb{R}^m$. (Note that $\varphi_{\balpha}$ is not necessarily in $H$)
\item[(iii)] $\balpha\mapsto \varphi_{\balpha,i}$ and $\balpha\mapsto \varphi_{\balpha,ij}$ are globally Lipschitz for all $1\leq i,j \leq m$. 
\end{itemize}
\end{assump}
Let $\mathcal{L}_{\balpha} = A + Df(\varphi_{\balpha})$ be the linearization of the drift in \eqref{main}, about $\varphi_{\balpha}$, and let $\mathcal{L}_{\balpha}^*$ be its adjoint. It follows from Assumption \ref{assump A} that $\mathcal{L}_{\balpha}$ generates a continuous semigroup $U_{\balpha}(t)$. That is, for any $z$ in the domain of $A$, $U_{\balpha}(t)\cdot z := v_t$, where $v_t$ satisfies the linear equation
\begin{equation}
\frac{dv_t}{dt} = \mathcal{L}_{\balpha}v_t,
\end{equation}
and $v_0 := z$. Now it follows from taking any directional derivative of \eqref{stationary equation} with respect to $\balpha$ that the linearized dynamics must always have a neutral eigenmode in directions tangential to the manifold $\lbrace \varphi_{\balpha} \rbrace_{\balpha \in \mathbb{R}^m}$. 

The next assumption essentially means that the linearized dynamics is stable in all other directions. Without this stability, in most circumstances the noise would quickly force the system away from the manifold $\lbrace \varphi_{\balpha} \rbrace_{\balpha \in \mathbb{R}^m}$. Because $\varphi_{\balpha} = \mathcal{T}_{\balpha}\cdot \varphi_0$, the following spectral gap property only needs to be verified for $\balpha = 0$.

\begin{assump}\label{Assump Spectral Gap}
The spectrum $\sigma(\mathcal{L}_{\balpha})$ of $\mathcal{L}_{\balpha} := A + Df(\varphi_{\balpha})$ is such that 
\[
\sigma(\mathcal{L}_{\balpha}) \subset \{\lambda\in \mathbb{C}: \mathfrak{Re}(\lambda)  \leq -b\} \cup\{0\},
\]
for some positive constant $b$, independent of $\alpha$. The eigenvalue $0$ is assumed to have multiplicity $m$, and the corresponding eigenvectors are spanned by $\lbrace \varphi_{\balpha,i} \rbrace_{1\leq i \leq m}$. 
\end{assump}

The above assumption implies that the essential spectrum of $\mathcal{L}_{\balpha}$ lies in the subset $\{\lambda\in \mathbb{C}: \mathfrak{Re}(\lambda)  \leq -b\}$. This means that $\mathcal{L}_{\balpha}$ is Fredholm of zero index, and therefore the kernel of the adjoint operator $\mathcal{L}^*_{\balpha}$ is $m$-dimensional. One can show \cite{Kapitula2013} that a basis $\lbrace \psi_{\balpha}^i \rbrace_{1\leq i \leq m}$ for the kernel of $L^*_{\balpha}$ can be chosen such that
\begin{align}
\langle \psi^i_{\balpha} , \varphi_{\balpha,j} \rangle &= 0 \text{ if }i\neq j \\
\langle \psi^i_{\balpha}, \varphi_{\balpha,i} \rangle &= 1.
\end{align}
The assumed invariance of $A$ and $f$ under the isometry $\mathcal{T}_{\balpha}$ implies that we can take
\begin{equation}
\psi^i_{\balpha} = \mathcal{T}_{\balpha} \psi^i_0.
\end{equation}

Let $P_{\balpha}: H \mapsto H$ be the following spectral projection operator that projects the kernel of $\mathcal{L}_{\balpha}^*$ onto the kernel of $\mathcal{L}_{\balpha}$, i.e.
\begin{align}
P_{\balpha}\cdot u &= \sum_{i=1}^m \langle \psi^i_{\balpha} , u \rangle  \varphi_{\balpha,i} \text{ and define }\\
\Pi_{\balpha} &= \mathcal{I} - P_{\balpha},\label{eq: Pi alpha definition}
\end{align}
where $\mathcal{I}: H \to H$ is the identity operator. Define the Resolvent Operator corresponding to $\Pi_{\balpha}$ to be, for $\lambda \in \mathbb{C}$,
\begin{equation}
\mathcal{R}_{\balpha}(\lambda) = (\lambda \mathcal{I} - \Pi_{\balpha})^{-1}.
\end{equation}
\begin{assump}
Suppose that there exists $M >0$ such that for all $\lambda \in \mathbb{C}$ with $\text{Re}(\lambda) \geq -b$ (here $b$ is the constant in Assumption \ref{Assump Spectral Gap}), and all $\balpha\in \mathbb{R}^m$,
\begin{equation}
\norm{ \mathcal{R}_{\balpha}(\lambda)} \leq M.
\end{equation}
\end{assump}
Note that once the above property is satisfied for $\balpha = 0$, it is satisfied for all $\balpha\in\mathbb{R}^m$, thanks to the fact that the latter is obtained from the former through the application of a smooth isometry. Define
\begin{equation}\label{eq: V definition}
V_{\balpha}(t) = U_{\balpha}(t) - P_{\balpha},
\end{equation}
and note that $V_{\balpha}(t) \cdot v = 0$ for any $v\in H$ such that
\[
\langle v , \psi^i_{\balpha} \rangle = 0 \text{ for all }1\leq i \leq m.
\]
It follows from the Gearhart-Pruss theorem (see \cite[Theorem 4.1.5]{Kapitula2013}) that there exists a constant $\mathfrak{c} \geq 1$ such that
\begin{equation}\label{eq: norm V bound}
\norm{V_{\balpha}(t)} \leq \mathfrak{c}\exp\big(-b t\big).
\end{equation}
This constant $\mathfrak{c}$ is independent of $\balpha$ thanks to the fact that any $\varphi_{\balpha}$ can be obtained from $\varphi_0$ through applying an isometry.

\begin{assump}
\begin{itemize}
\item [(i)] For $1\leq i,j,k \leq m$, the derivatives  $[\partial /\partial \alpha_j  ] \psi^i_{\balpha}$,  $[\partial^2 /\partial \alpha_j \partial \alpha_k ] \psi^i_{\balpha}$ exist (the derivatives being taken in the norm of the space $H$) and are all in the space $H$. They are written as (respectively) $\lbrace \psi^i_{\balpha,j}, \psi^i_{\balpha,jk} \rbrace$.
\item[(ii)] $\balpha\mapsto \psi^i_{\balpha}$, $\balpha\mapsto \psi^i_{\balpha,j}$ and $\balpha\mapsto \psi^i_{\balpha,jk}$ are all globally Lipschitz.
\item[(iii)]  Integration by parts holds i.e. $  \langle \psi^i_{\balpha,j} , \varphi_{\balpha}\rangle + \langle \psi^i_{\balpha} , \varphi_{\balpha ,j} \rangle = 0$ and $ \langle \psi^i_{\balpha,jk} , \varphi_{\balpha}\rangle + \langle \psi^i_{\balpha,j} , \varphi_{\balpha ,k} \rangle = 0$.
\item[(iv)] $\balpha \mapsto A^*\psi^i_{\balpha}$, $\balpha \mapsto A^*\psi^i_{\balpha,j}$ and $\balpha \mapsto A^*\psi^i_{\balpha,jk}$ are globally Lipschitz for each $1\leq i,j,k \leq m$. 
\end{itemize}
\end{assump}

\begin{assump}
The multiplicative noise operator is assumed to have the properties, for a constant $C_B > 0$,
\begin{align*}
B &: [0,\infty) \times E \to \mathcal{L}( H, H) \\
\norm{B(t, x) - B(t,y)}_{\mathcal{L}} &\leq C_B \norm{x-y} \text{ for all }t\geq 0 \text{ and }x,y \in E \\
\norm{B(t,x)}_{\mathcal{L}} &\leq C_B
\end{align*}
\end{assump}

The following assumption is needed to obtain the exponential moment necessary for the exit-time bound in Section \ref{Section Long Time Stability}.
\begin{assump}\label{assump exponential moment}
For any $t > s$ and any $x \in E$, $U_{\balpha}(t-s) B(s,x)$ is a Hilbert-Schmidt operator, with Hilbert-Schmidt norm (as defined in \eqref{eq: HS norm}) upperbounded by 
\begin{equation}\label{eq: HS upperbound}
\sup_{\balpha \in \mathbb{R}^m}\int_0^{T_0}(t-s)^{-2 \mathfrak{z}} \sup_{x\in E}\norm{ U_{\balpha}(t-s)B(s,x)}_{HS}^2 ds< C_{HS} ,
\end{equation}
for some $\mathfrak{z} \in (0 , 1/2)$ and constant $C_{HS} < \infty$, where
\[
T_0 = \frac{\log 4\mathfrak{c}}{b}.
\]
\end{assump}

\section{Examples}
\label{examples}
This framework applies to a huge range of stochastically-forced traveling waves, stationary patterns, spiral waves, and complex spatially-distributed patterns in neural fields. We outline some examples: further examples can be found in \cite[Section 4.6]{Kapitula2013}.
\subsection{Stationary Bump in Neural Fields on $\mathbb{S}^1$}

We consider a one-dimensional neural field on $\mathbb{S}^1$. This model was originally developed to model orientation selectivity in the primate visual cortex \cite{Medathati2017}. To the best of this author's knowledge, Kilpatrick and Ermentrout were the first to prove that the bump is linearly stable \cite{Kilpatrick2013a}. They also studied its wandering under the effect of space-time noise over long timescales. This analysis was continued in \cite{Kilpatrick2015,Bressloff2015a,MacLaurin2020a}. In fact the variational phase SDE and long-time bound on the escape probability have already been determined in \cite{MacLaurin2020a}. In some circumstances this system can be written as a gradient flow, and exact analytical formulae are available \cite{Bressloff2019e}.

The evolution is described by the neural field equation on the ring $\mathbb{S}^1$:
\begin{eqnarray}
  \tau \frac{\partial u(\theta,t)}{\partial t}&=& -u(\theta,t)+\int_{-\pi}^{\pi} J(\theta-\theta')f(u(\theta',t))d\theta' 
  \label{wc0}
\end{eqnarray}
where $u(\theta,t)$ denotes the activity at time $t$ of a local population of cells with direction
preference $\theta \in [-\pi,\pi)$, $J(\theta-\theta')$ is the strength of synaptic weights between cells with
direction preference
$\theta'$ and $\theta$. (Most applications of the ring model take $\theta \in [0,\pi]$ and interpret $\theta$ as the orientation preference of a population of neurons in primary visual cortex, see for example \cite{Medathati2017,Bressloff2002}.) 
The weight distribution is a 2$\pi$-periodic and even function of $\theta$ and thus has the cosine series expansion
\begin{equation}
\label{J} 
J(\theta)=\sum_{n=0}^NJ_n\cos (n\theta) .
\end{equation}
For analytical simplicity, we assume that there are a finite number of terms in the series expansion.
Finally, the firing rate function is taken to be a sigmoid $F(u)=\big(1+e^{-\gamma(u-\kappa)} \big)^{-1}$ with gain $\gamma$ and threshold $\kappa$. To fit the dynamics into the formalism of the previous section, we take $A\cdot u = - u$, and $f(u)(\theta) = \int_{-\pi}^{\pi} J(\theta-\theta')f(u(\theta',t))d\theta'$.  Indeed the operators $A$ and $f$ are both bounded and Lipschitz over the Hilbert space $L^2(\mathbb{S}^1)$.  

The dynamics in \eqref{wc0} is invariant under the translation by $\theta$ operator: $\big(\mathcal{T}_\theta \cdot u \big)(\theta') := u(\theta' - \theta)$, with $\theta' - \theta $ taken modulo $\mathbb{S}^1$. Ermentrout and Kilpatrick \cite{Kilpatrick2013} proved the existence and stability of a family of stationary bump solutions $\lbrace U_{\theta} \rbrace_{\theta \in \mathbb{S}^1}$, and we therefore define $\varphi_{\balpha} = U_{\balpha \mod \mathbb{S}^1}$, and take $m=1$. Since the operator $A$ is bounded, the constant $\mathfrak{c}$ is $1$ \cite{MacLaurin2020a}.

A key difference between the above neural field equation and reaction-diffusion systems it that the neural field equation does not have a Laplacian, which works to smooth spatial irregularities. This means that in order that the stochastic equation is well defined, we require that the stochastic integral $\int_0^t B(u_s)dW_s$ belongs to the Hilbert space $H$. For this to be the case, we require that for any orthonormal basis $\lbrace e_j \rbrace_{j\geq 1}$ of $H$,
\[
\sup_{u\in E}\sum_{j=1}^\infty \big\langle e_j , B(u)e_j \big\rangle < \infty.
\]
In the terminology of \cite{DaPrato2014}, $B(u)$ is a trace class operator. Some neural field equations include a Laplacian, see for instance \cite{Liley2014,Coombes2014,Mukta:2017jm}.

\subsection{Traveling fronts in neural field equations}
\label{sec:traveling fronts}
Neural field equations taking values in $\mathbb{R}$ are known as the Wilson-Cowan equations \cite{Wilson1972}. They take the form
\begin{equation}
\label{NF}
\partial_tu_t(x) = -u_t(x) + \int_{\R}w(x - y)F(u_t(y))dy, \quad t\geq 0,\ x\in\R,
\end{equation}
where $w\in \mathcal{C}(\R)\cap L^1(\R)$ is the connectivity function, and $F:\R \to \R$ is a smooth and bounded sigmoid function (known as the nonlinear gain function).
It is known (see \cite{Ermentrout1993} for example) that under some conditions on the functions $w$ and $F$ (in particular that there exist precisely three solutions to the equation $x=F(x)$ at $0, a$ and $1$ with $0<a<1$), then there exists a unique (up to translation) function $\hat{u}\in\mathcal{C}^\infty(\R)$ and speed $c\in\R$ such that $u_t(x) = \hat{u}(x - ct)$ is a solution to \eqref{NF}, where  $\hat{u}$ is such that
\[
 \lim_{x\to-\infty} \hat{u}(x) =0, \qquad \lim_{x\to\infty} \hat{u}(x) = 1,
\]
so that $\hat{u}$ is indeed a wave front.   Note that in this case $\hat{u}$ itself is not in $L^2(\R)$, but it can be shown that all derivatives of $\hat{u}$ are bounded and in $L^2(\R)$.

Substituting $\hat{u}(x - ct)$ into \eqref{NF}, we see that $\hat{u}$ is such that $0= A\hat{u} + f(\hat{u})$,
where $Au := cu' $ and $f(u) = -u + w*F(u)$, and $*$ denotes convolution as usual.
Moreover, due to translation invariance, we have that $\hat{u}_\alpha := \hat{u}(\cdot + \alpha)$ is also such that
\begin{align}
\label{deterministic solution}
0&= A\hat{u}_\alpha + f(\hat{u}_\alpha), \qquad\alpha\in\R.
\end{align}
We must thus interpret \eqref{NF} in the moving co-ordinate frame, i.e. writing $\zeta = x-ct$,
\begin{equation}
\label{NF}
\partial_t\hat{u}_t( \zeta ) = c \partial_{\zeta}u_t(\zeta)-u_t( \zeta ) + \int_{\R}w(\zeta - y)F(\hat{u}_t(y))dy, \quad t\geq 0,\ x\in\R,
\end{equation}
The traveling front solutions are fixed points of the above equation, and the family of isometries is translation. We are thus in a specific situation of the general setup described in the previous section, with $H=L^2(\R)$ and $\varphi_{\balpha} :=\hat{u}_\alpha$. The spectral gap property has been proved in \cite{Lang2016}.

\subsection{Traveling Waves in Reaction-Diffusion Systems}
Consider the one-dimensional reaction-diffusion system
\begin{equation}
du_t = \lbrace \tilde{A} u_t + \mathcal{W}'(u_t) \rbrace dt + B(u_t)dW_t,
\end{equation}
with $\tilde{A}$ the Laplacian $\frac{\partial^2}{\partial x^2}$, and $\mathcal{W}$ a potential function. In many circumstances such systems support traveling fronts. See the discussion in \cite[Section 4.2.2]{Kapitula2013}. One must take care to work in the co-moving frame to apply the formalism of this paper (as discussed in Section \ref{sec:traveling fronts}, see also Remark \ref{Important Remark}): then the traveling wave solutions constitute a manifold of fixed points (invariant under spatial translation). We thus write $A := \tilde{A} + c\frac{\partial}{\partial\zeta}$, and in the moving frame the dynamics is of the form
\begin{equation}
du_t = \lbrace A u_t + \mathcal{W}'(u) \rbrace dt + \tilde{B}(t,u_t)dW_t,
\end{equation}
 and $\tilde{B}(t , u_t)w := B( \tilde{w} )$, $\tilde{w}(x) = w(x+ct)$ for $w\in E$. See \cite{Cartwright2019,Hamster2020} for a discussion of the stochastic Nagumo equation.

\subsection{Traveling pulses in neural fields}
One can modify the classical neural field equation \eqref{NF} to produce traveling pulse solutions in the following way.  Indeed consider the system
\begin{equation}
\label{NFwA}
\begin{cases}
\partial_tu_t = -u_t + \int_{\R}w(\cdot - y)F(u_t(y))dy -v_t, \quad t\geq 0\\
\partial_tv_t =  \theta u_t - \beta v_t,
\end{cases}
\end{equation}
where as above $F:\R \to \R$ is a smooth and bounded sigmoid function, $w\in \mathcal{C}(\R)\cap L^1(\R)$ and $\theta>0, \beta\geq0$ are some constants with $\theta<<\beta$ . This is called the neural field equation with adaptation (see for example \cite[Section 3.3]{Bressloff2012} for a review).   This time we look for a solution to \eqref{NFwA} of the form $(u_t, v_t) = (\hat{u}(\cdot-ct), \hat{v}(\cdot-ct))$ for some $c\in\R$, such that $\hat{u}(x)$ and $\hat{v}(x)$ decay to zero as $x\to\pm\infty$.  Substituting this into \eqref{NFwA}, we are thus looking for a solution to the equation
\begin{equation}
\label{NFwAvector}
cU'(x) = \left(\begin{array}{cc} -1 &- 1\\ \theta  & -\beta\end{array}\right)U(x)+ f(U)(x), \quad x\in \R,
\end{equation}
where $U(x) = (\hat{u}(x), \hat{v}(x))$, and $f(U)(x) := (w*F(\hat{u})(x), 0)^T$, for all $x\in\R$.

It can be shown (see \cite[Section 3.1]{Pinto2001} or \cite{Faye2015})
that there exists (again under some conditions on the parameters) a smooth function $U := (\hat{u}, \hat{v})\in [L^2(\R)]^2$ and speed $c\in\R$ such that $U$ is a solution to \eqref{NFwAvector}.  Moreover $\hat{u}$ and $\hat{v}$ are both smooth functions whose derivatives are all bounded and in $L^2(\R)$.  Thus, again by translation invariance we have that $U_\alpha:= U(\cdot + \alpha)\in [L^2(\R)]^2$ is a solution to
\[
AU_\alpha + f(U_\alpha) = 0
\]
for all $\alpha\in\R$, where 
\[
AU: = cU' - \left(\begin{array}{cc} -1 &- 1\\ \theta  & -\beta\end{array}\right)U, \quad\forall U\in [L^2(\R)]^2.
\]
Once again we are thus in a specific situation of the general setup described in Section \ref{General Setting}, this time with $H=[L^2(\R)]^2$ and $\varphi_{\balpha} :=U_\alpha$.   Since $\hat{u}(x)\to0$ as $x\to\pm\infty$, we say that the solution is a traveling pulse. The stability of the traveling pulse has been proved in \cite{Ermentrout2014}. Other types of neural field models also support traveling pulses, such as \cite{Keane2015}. See \cite{Eichinger2020} for the development of a phase decomposition broadly similar to the methods outlined in this paper.

\subsection{Neural Field Patterns on Higher Dimensional Domains}
In the one-dimensional neural field model of the previous section, neurons are grouped according to their orientation selectivity, which takes on values between $-\pi / 2$ and $\pi / 2$. More sophisticated neural field models can involve patterns with more degrees of freedom \cite{Bressloff2012,Coombes2014,Carroll2016}. One example is the celebrated explanation of hallucinations using a neural field model that is invariant under three types of group action: rotation, reflection in the plane, and a shift-twist action \cite{Bressloff2002}. If one were to impose space-time noise on this model (doing this has excellent biophysical motivation, because brain signals are typically very noisy), then one could easily observe a rich range of metastable phenomena over long time periods (applying the results of Section 6). Another recent example is the analysis of the wandering of bumps of neural activity over the sphere in \cite{Visser2017,Bressloff2019b}: this has two degrees of freedom. The ergodic results of Section 6 can be applied to the wandering of a bump of activity over the sphere. One must use spherical polar co-ordinates $(\theta,\phi)$, and make sure to identify the points $(-\theta,\phi)$ and $(\theta,\phi)$, and one must identify $(\theta + 2k\pi , \phi + 2l\pi)$ with $(\theta,\phi)$.

\subsection{The Scalar Viscous Conservation Law}
This is the system
\begin{equation}
du_t = \lbrace A u_t + \partial_x f(u_t) \rbrace dt + B(u)dW_t,
\end{equation}
with $A$ the Laplacian. See the discussion in \cite[Section 4.4]{Kapitula2013}.

\subsection{The Parametrically-Forced Nonlinear Schrodinger Equation}

See \cite[Section 4.5]{Kapitula2013}.

\subsection{Spiral Waves in Reaction Diffusion Systems}

Spiral waves are pervasive in non-equilibrium reaction-diffusion systems \cite{Barkley1992,Barkley1994,Sandstede1997,Vanag2019}. Spatially-extended oscillations are also present in neural field equations \cite{Folias2005}. \cite{Nakao2014} have identified an equation for the leading order diffusive flux of the phase of oscillations in reaction-diffusion systems, a result that is consistent with this paper. Consider for examples the two-species model in \cite{Barkley1992}
\begin{align*}
\frac{\partial u}{\partial t} &= \nabla^2 u + \epsilon^{-1}u(1-u)\big\lbrace u - (v+b) / a \big\rbrace \\
\frac{\partial v}{\partial t} &= \delta \nabla^2 v + u-v.
\end{align*}
The domain is a circle of radius $R$.  \cite{Barkley1992} demonstrate that there exist stable spiral wave solutions to the above system. In a co-rotating reference frame, these solutions are fixed points. In this co-rotating frame, there is a manifold of solutions $\lbrace \varphi_{\theta} \rbrace_{\theta \in \mathbb{S}^1}$. One solution can be obtained from another by applying a rotation isometry. If one imposes space-time white noise on the above system, then the resulting system will fit the requirements of Section 2. 

Furthermore, the formalism of Section 6 could be used to determine the long-time average phase shift induced by noise correlations. These results parallel existing results for the long-time average phase shift of finite-dimensional stochastic oscillators \cite{Giacomin2018}.
%
%
%

\section{Definition of the Stochastic Phase}
In this section we outline two different phase definitions: the variational phase, and the isochronal phase. As explained in the introduction, one of the main goals of this paper is to determine elegant and useful stochastic differential equations for the phase;  indeed the variational phase and isochronal phase each have particular merits. The variational phase is obtained by insisting that the amplitude is orthogonal to the neutral eigenmodes of the adjoint operator. The isochronal phase is the limiting point on the manifold $\lbrace \varphi_{\balpha} \rbrace_{\balpha \in \mathbb{R}^m}$ that the system would  converge to in the absence of noise. The chief advantages of the variational phase are (i) its stochastic dynamics admits a more tractable analytic expression, (ii) it can be easily employed to obtain powerful exponential bounds on the probability of the system leaving a close neighborhood of the manifold (as performed in Section \ref{Section Long Time Stability}) and (iii) less regularity assumptions on $f$, $A$ and $B$ are required for the variational phase SDE than the isochronal phase SDE (note the additional assumptions at the start of Section \ref{Section Isochronal}). The chief advantages of the isochronal phase are (i) for finite-dimensional oscillators, the isochronal phase is the phase definition most preferred by experts, and so it is natural to search for its analog in our infinite-dimensional case, and (ii) it can be used to accurately predict the average occupation times of the system as it wanders close to the manifold over very long periods of time (as performed in Section \ref{Section Ergodicity}).
\subsection{Variational Phase SDE}
\label{sec:stoch wave}
The variational phase $\bbeta_t$ is defined to be such that (i) it is continuous for all $t < \tau$ ($\tau$ is a stopping time defined in \eqref{eq: stopping time}) and (ii) for all $t < \tau$, it exactly solves the identities, for $1\leq i \leq m$,
\begin{align}\label{eq: G i definition}
\mathcal{G}_i(u_t,\bbeta_t) &= 0 \text{ where }\mathcal{G}_i: E \times \mathbb{R}^m \mapsto \mathbb{R} \text{ is such that }\\
\label{zero first deriv}
\mathcal{G}_i(z,\balpha) &:= \langle z - \varphi_{\balpha},\psi^i_{\balpha}\rangle .
\end{align}
This phase definition agrees with our definition in \cite{MacLaurin2020a} for `stochastic neural bumps' in terms of a weighted Hilbert space: in this paper we determined an equation of the form \eqref{eq: G i definition} by defining the phase to minimize a potential weighted by the ratio of the eigenvectors (see also \cite{kruger2014}). The definition is different from our definition in \cite{inglis2016general} in two respects: (i) it is multi-dimensional, and (ii), in \cite{inglis2016general}, instead of the eigenvectors $\lbrace \psi^i_{\balpha} \rbrace$ of the adjoint operator $\mathcal{L}_{\balpha}^*$, we have $ \varphi_{\balpha,i}$ (the eigenvector of $\mathcal{L}_{\balpha}$). To leading order in $\norm{ u_t - \varphi_{\bbeta_t}}^2$, the phase definitions in \cite{inglis2016general} and \eqref{eq: G i definition} are equivalent, and either could be used to obtain accurate long-time stability estimates. Upto linear order in $\epsilon$, this definition agrees with the definitions in \cite{Cartwright2019,Hamster2020,Hamster2020b}.

Standard theory \cite[Lemma 4.3.3] {Kapitula2013} dictates that \eqref{eq: G i definition} has a unique solution $\bbeta_t$ as long as $u_t$ is close enough to the manifold $\lbrace \varphi_{\balpha} \rbrace_{\balpha\in \mathbb{R}^m}$. In Lemma \ref{Lemma Beta SDE}, we will prove that $\bbeta_t$ is uniquely well-defined for all times upto $\tau$, and we will outline a precise stochastic differential equation for $\bbeta_t$. However before we do this, we start with some informal calculations to motivate the definition of the stochastic phase. %
Notice first that our initial condition is such that \eqref{eq: G i definition} is satisfied exactly (for $t=0$). Now define $\mathcal{M}(z,\balpha)$ to be the $m\times m$ square matrix with elements
\begin{equation}
\mathcal{M}_{ij}(z,\balpha) = -\frac{\partial}{\partial \alpha_j}\mathcal{G}_i(z,\balpha).
\end{equation}
It follows from the implicit function theorem that \eqref{eq: G i definition} is solvable for the phase in some neighborhood of $(u_t,\bbeta_t)$ as long as the matrix $\mathcal{M}(u_t,\bbeta_t)$ is invertible. We therefore define the stopping time
\begin{equation}\label{eq: stopping time}
\tau = \inf\big\lbrace t\geq 0: \det\big( \mathcal{M}(u_t,\bbeta_t) \big) =0 \big\rbrace,
\end{equation}
and we assume that $t < \tau$, so that a local solution for $\beta_t$ in terms of $u_t$ is possible.

Since we are assuming that $\langle \varphi_{\bbeta} , \psi_{\bbeta}^i \rangle$ is invariant under $\bbeta$, we find that
\begin{align}
\mathcal{M}_{ij}(z,\balpha) &= -\langle z , \psi^i_{\balpha , j} \rangle \\
&=- \langle z - \varphi_{\balpha} , \psi^i_{\balpha,j} \rangle - \langle \varphi_{\balpha} , \psi^i_{\balpha,j} \rangle.
\end{align}
Our integration by parts assumption implies that $\langle \varphi_{\balpha} , \psi^i_{\balpha,j} \rangle = -  \langle \varphi_{\balpha,j} , \psi^i_{\balpha} \rangle = -\delta(i,j) $, by assumption. We can thus write
\begin{equation}\label{eq: M definition}
\mathcal{M}_{ij}(z,\balpha) = \delta(i,j) - \langle z - \varphi_{\balpha} , \psi^i_{\balpha,j} \rangle .
\end{equation}
In this above form, it is clear that as long as $\| u_t - \varphi_{\bbeta_t} \| $ is sufficiently small, $\mathcal{M}(u_t,\bbeta_t)$ is always invertible.

One can guess the dynamics of $\bbeta_t$ by first assuming that $\bbeta_t$ satisfies an SDE of the form
\begin{equation}\label{eq: beta SDE definition}
d\bbeta_t = \mathcal{V}(u_t, \bbeta_t) dt + \epsilon \mathcal{Y}(t,u_t,\bbeta_t)dW_t,
\end{equation}
for functions $\mathcal{V}: \mathbb{R}^+ \times E \times \mathbb{R}^m \to \mathbb{R}^m$ and $\mathcal{Y}: \mathbb{R}^+\times E \times \mathbb{R}^m \to \mathcal{L}(H, \mathbb{R}^m)$ to be determined below. As explained in \cite{inglis2016general,MacLaurin2020a}, one can then formally expand out the identity $d\mathcal{G}_i(u_t, \bbeta_t) = 0$ and (i) insist that the stochastic terms are zero to determine $\mathcal{Y}$, and then (ii) insist that the drift terms (i.e. the terms of finite variation) are zero, and thus determine $\mathcal{V}$. To this end, using Ito's Lemma, 
\begin{equation}
d\mathcal{G}_{i,t} = \langle du_t , \psi^i_{\bbeta_t} \rangle+ \sum_{j=1}^m \frac{\partial \mathcal{G}_i}{\partial \beta_t^j}d\beta^j_t + \frac{1}{2}\sum_{j,k=1}^m \frac{\partial^2 \mathcal{G}_i}{\partial \beta_t^j\partial \beta_t^k}d\beta^j_t d\beta^k_t + \sum_{j=1}^m\langle du_t , \psi^i_{\bbeta_t,j} \rangle d\beta^j_t , \label{eq: mathcal G i t}
\end{equation}
where the respective covariations of the processes are written as $d\beta^j_t d\beta^k_t$ and $du_t d\beta^j_t$. Now if $u_t$ were in the domain of $A$ and $Au_t \in H$, then using the fact that $A\varphi_{\bbeta_t} + f(\varphi_{\bbeta_t}) = 0$, it would hold that
\begin{align}
\langle du_t , \psi^i_{\bbeta_t} \rangle &= \langle \lbrace Au_t + f(u_t) \rbrace dt + \epsilon B(t,u_t)dW_t , \psi^i_{\bbeta_t} \rangle \\
 &= \langle \lbrace Au_t - A\varphi_{\bbeta_t} + f(u_t) - f(\varphi_{\bbeta_t}) \rbrace dt + \epsilon B(t,u_t)dW_t , \psi^i_{\bbeta_t} \rangle \\
 &= \langle u_t - \varphi_{\bbeta_t} , A^* \psi^i_{\bbeta_t} \rangle dt+ \langle f(u_t) - f(\varphi_{\bbeta_t}) , \psi^i_{\bbeta_t} \rangle dt + \epsilon \langle B(t,u_t) dW_t , \psi^i_{\bbeta_t} \rangle. \label{eq: dut psi t}
\end{align}
In deriving the last expression, we assumed that $u_t$ is in the domain of $A$. If $u_t$ is not in the domain of $A$, then \eqref{eq: dut psi t} is still well-defined (our assumptions dictate that $f(u_t) - f(\varphi_{\bbeta_t}) \in H$ and $A^* \psi^i_{\bbeta_t} \in H$), and we will see in the next section that \eqref{eq: dut psi t} is in fact the correct expression to use.

Matching the stochastic terms (the coefficients of $dW_t$) in \eqref{eq: mathcal G i t}, we find that
\begin{equation}
 \epsilon \langle B(t,u_t) dW_t , \psi^i_{\bbeta_t} \rangle - \epsilon \sum_{j=1}^m \mathcal{M}_{ij}(u_t,\bbeta_t)\mathcal{Y}_j(t,u_t,\bbeta_t)dW_t = 0.
\end{equation}
Inverting this equation, we find that the linear operator $\mathcal{Y}_j(t,u_t,\bbeta_t)$ must be such that for each $z\in H$,
\begin{align}\label{eq: Y definition}
\mathcal{Y}_i(t,u_t,\bbeta_t) \cdot z &=\sum_{j=1}^m \mathcal{N}_{ij}(u_t,\bbeta_t) \langle B(t,u_t) z,\psi^j_{\bbeta_t} \rangle \text{ where } \\
\mathcal{N}(u_t,\bbeta_t) &= \mathcal{M}(u_t,\bbeta_t)^{-1} \text{ and }\mathcal{N}(u_t,\bbeta_t) = \big( \mathcal{N}_{ij}(u_t,\bbeta_t) \big)_{1\leq i,j \leq m},
\end{align}
noting that $\mathcal{M}(u_t,\bbeta_t)^{-1}$ is the matrix inverse of $\mathcal{M}(u_t,\bbeta_t)$. It is immediate from the definition of the stopping time that $\mathcal{M}(u_t,\bbeta_t)$ is invertible for $t < \tau$. We thus find that the covariation terms must have the form (using standard theory for stochastic integrals with respect to infinite-dimensional Wiener Processes \cite[Chapter 4.3]{DaPrato2014}),
\begin{align}
d\beta^j_t d\beta^k_t &= \epsilon^2 \sum_{p,q=1}^m \mathcal{N}_{jp}(u_t,\bbeta_t)\mathcal{N}_{kq}(u_t,\bbeta_t) \langle B^*(t,u_t) \psi_{\bbeta_t}^p , B^*(t,u_t) \psi^q_{\bbeta_t} \rangle dt \label{eq: d beta j d beta k}  \\
\langle du_t , \psi^i_{\bbeta_t,j} \rangle d\beta^j_t  &= \epsilon^2 \sum_{p=1}^m \mathcal{N}_{jp}(u_t,\bbeta_t) \langle B^*(t,u_t)\psi^p_{\bbeta_t} , B^*(t,u_t) \psi^i_{\bbeta_t,j} \rangle dt. \label{eq: d u j d beta k} 
\end{align}

The above terms do not directly depend on $\mathcal{V}(u_t,\bbeta_t)$, which means that we can easily solve \eqref{eq: mathcal G i t} for $\mathcal{V}(u_t,\bbeta_t)$ by matching all of the coefficients of $dt$ terms. Observe that
\begin{align}\label{eq: G second derivative}
\frac{\partial^2 \mathcal{G}_i}{\partial \alpha^j \partial \alpha^k} = \langle z , \psi^i_{\balpha,jk}\rangle =\langle z - \varphi_{\balpha} , \psi^i_{\balpha,jk} \rangle + \langle \varphi_{\balpha} , \psi^i_{\balpha,jk} \rangle = \langle z - \varphi_{\balpha} , \psi^i_{\balpha,jk} \rangle - \langle \varphi_{\balpha,j} , \psi^i_{\balpha,k} \rangle ,  
\end{align}
using the integration by parts formula. We find that
\begin{multline}
-\sum_{j=1}^m \mathcal{M}_{ij}(u_t,\bbeta_t) \mathcal{V}_j(u_t,\bbeta_t)  +  \epsilon^2 \sum_{j,p=1}^m \mathcal{N}_{jp}(u_t,\bbeta_t) \langle B^*(t,u_t)\psi^p_{\bbeta_t} , B^*(t,u_t) \psi^i_{\bbeta_t,j} \rangle\\
+  \frac{\epsilon^2}{2} \sum_{j,k,p,q=1}^m \langle u_t , \psi^i_{\bbeta_t,jk} \rangle \mathcal{N}_{jp}(u_t,\bbeta_t)\mathcal{N}_{kq}(u_t,\bbeta_t) \langle B^*(t,u_t) \psi_{\bbeta_t}^p , B^*(t,u_t) \psi^q_{\bbeta_t} \rangle \\
+ \langle u_t - \varphi_{\bbeta_t} , A^* \psi^i_{\bbeta_t} \rangle + \langle f(u_t) - f(\varphi_{\bbeta_t}) , \psi^i_{\bbeta_t} \rangle = 0.
\end{multline}
Inverting the matrix $\mathcal{M}(u_t,\bbeta_t)$, we thus find that for $1\leq r \leq m$,
\begin{multline}\label{eq: V ut beta t 0}
\mathcal{V}_r(t,u_t,\bbeta_t) = \sum_{i=1}^m \mathcal{N}_{ri}(u_t,\bbeta_t) \bigg\lbrace  \epsilon^2 \sum_{j,p=1}^m \mathcal{N}_{jp}(u_t,\bbeta_t) \langle B^*(t,u_t)\psi^p_{\bbeta_t} , B^*(t,u_t) \psi^i_{\bbeta_t,j} \rangle\\
+  \frac{\epsilon^2}{2} \sum_{j,k,p,q=1}^m \langle u_t , \psi^i_{\bbeta_t,jk} \rangle \mathcal{N}_{jp}(u_t,\bbeta_t)\mathcal{N}_{kq}(u_t,\bbeta_t) \langle B^*(t,u_t) \psi_{\bbeta_t}^p , B^*(t,u_t) \psi^q_{\bbeta_t} \rangle \\
 \langle u_t - \varphi_{\bbeta_t} , A^* \psi^i_{\bbeta_t} \rangle + \langle f(u_t) - f(\varphi_{\bbeta_t}) , \psi^i_{\bbeta_t} \rangle\bigg\rbrace .
\end{multline}
Now
\begin{align*}
\big\langle u_t - \varphi_{\bbeta_t} , A^* \psi^i_{\bbeta_t} \big\rangle + \big\langle Df(\varphi_{\bbeta_t})\cdot (u_t - \varphi_{\bbeta_t}) , \psi^i_{\bbeta_t} \big\rangle = \big\langle u_t - \varphi_{\bbeta_t} , \mathcal{L}^*_{\bbeta_t} \psi^i_{\bbeta_t} \big\rangle = 0,
\end{align*}
since by definition $\psi^i_{\bbeta_t}$ is an eigenvector of $\mathcal{L}^*_{\bbeta_t}$. We thus find that 
\begin{multline}\label{eq: V ut beta t}
\mathcal{V}_r(t,u_t,\bbeta_t) = \sum_{i=1}^m \mathcal{N}_{ri}(u_t,\bbeta_t) \bigg\lbrace  \epsilon^2 \sum_{j,p=1}^m \mathcal{N}_{jp}(u_t,\bbeta_t) \langle B^*(t,u_t)\psi^p_{\bbeta_t} , B^*(t,u_t) \psi^i_{\bbeta_t,j} \rangle\\
+  \frac{\epsilon^2}{2} \sum_{j,k,p,q=1}^m \langle u_t , \psi^i_{\bbeta_t,jk} \rangle \mathcal{N}_{jp}(u_t,\bbeta_t)\mathcal{N}_{kq}(u_t,\bbeta_t) \langle B^*(t,u_t) \psi_{\bbeta_t}^p , B^*(t,u_t) \psi^q_{\bbeta_t} \rangle \\
+ \langle f(u_t) - f(\varphi_{\bbeta_t}) -Df(\varphi_{\bbeta_t})\cdot (u_t - \varphi_{\bbeta_t}), \psi^i_{\bbeta_t} \rangle\bigg\rbrace .
\end{multline}
\subsubsection{Rigorous Definition of the Variational Phase SDE}

In the previous section, we guessed the form that the phase SDE should take by matching coefficients in the expression \eqref{eq: G i definition}. We now rigorously prove that this informal derivation (i) defines a unique stochastic process $\bbeta_t$, and (ii) \eqref{eq: G i definition} is satisfied. We recall the definitions of the functions  $\mathcal{V}:\mathbb{R}^+\times E \times \mathbb{R}^m \to \mathbb{R}^m$, $\mathcal{V} = (\mathcal{V}_i)_{1\leq i \leq m}$ and $\mathcal{Y}:\mathbb{R}^+ \times  E \times \mathbb{R}^m \to \mathcal{L}(H, \mathbb{R}^m)$ in \eqref{eq: Y definition} and \eqref{eq: V ut beta t}.

Now define $\bbeta_t$ to satisfy the $\mathbb{R}^m$-valued SDE
\begin{equation}\label{eq: beta SDE definition 2}
d\bbeta_t = \mathcal{V}(t,u_t, \bbeta_t) dt + \epsilon \mathcal{Y}(t,u_t,\bbeta_t)dW_t,
\end{equation}
with initial condition $\bbeta_0 = \bar{\bbeta}$, for all times $t$ upto the stopping time $\tau$. We are going to see that this definition is consistent with our previous definition of $\bbeta_t$ in \eqref{eq: G i definition}. Notice that the SDE for $\bbeta_t$ depends on the noise $W_t$ and solution $u_t$ of the original system. It is therefore essential to our argument that there exists a \textit{strong} solution to the SDE (see \cite[Chapter 5]{Karatzas1991} for a definition of a strong solution). In other words we need more than just an identification of the probability law of $\bbeta_t$; we also require that it is coupled in the same space as $u_t$ and $W_t$.

\begin{lemma}\label{Lemma Beta SDE}
There exists a unique strong solution $\bbeta_t$ to the SDE in \eqref{eq: beta SDE definition 2} for all times $t < \tau$. Furthermore this solution is such that, for all $t< \tau$,
\begin{align}\label{eq: G i definition 2}
\mathcal{G}_i(u_t,\bbeta_t) &= 0 \text{ where } \\
\label{zero first deriv 2}
\mathcal{G}_i(z,\balpha) &= \langle z - \varphi_{\balpha},\psi^i_{\balpha}\rangle .
\end{align}
\end{lemma}
\begin{proof}
The existence and uniqueness of the strong solution $\bbeta_t$ follows straightforwardly from the fact that the coefficient functions $\mathcal{V}$ and $\mathcal{Y}$ are locally Lipschitz in $\bbeta_t$ (see the proof in the one-dimensional case in \cite{inglis2016general}). Also, the operator $\mathcal{Y}$ is evidently Hilbert-Schmidt.

To prove \eqref{eq: G i definition 2}, one might wish to try to find an infinite-dimensional Ito's Lemma \cite{DaPrato2014} (this is the change-of-variable formula for stochastic differential equations, analogous to the chain rule of differential calculus). However the possible unboundedness of the operator $A$ complicates any easy generalization of Ito's Lemma to infinite dimensions. We thus instead take care to rigorously prove this; adapting the standard proof of Ito's Lemma (see \cite{DaPrato2014} and \cite[Theorem 4.17]{Karatzas1991}) to our setting. The trick to handling the unbounded operator $A$ is to instead work with its adjoint acting on $\psi^i_{\bbeta_t}$: the smoothness of the manifold $(\psi^i_{\balpha})_{\balpha\in \mathbb{R}^m}$ ensures that this is well-behaved.

Write $v_t = u_t - \varphi_{\bbeta_t}$. It follows from Ito's Lemma that, substituting the identity $A\varphi_{\bbeta_t} + f(\varphi_{\bbeta_t}) = 0$,
\begin{align}
dv_t =& \big(Au_t + f(u_t)\big) dt +\epsilon B(t,u_t)dW_t -\sum_{i=1}^m \varphi_{\bbeta_t , i}d\bbeta^i_t - \frac{1}{2}\sum_{j,k=1}^m \varphi_{\bbeta_t,jk}d\beta^j_t d\beta^k_t \\
=& \big( \mathcal{L}_{\bbeta_t} v_t + f(u_t) - f(\varphi_{\bbeta_t}) -Df(\varphi_{\bbeta_t})\cdot v_t \big) dt +\epsilon B(t,u_t)dW_t -\sum_{i=1}^m \varphi_{\bbeta_t , i}d\bbeta^i_t\nonumber \\  &- \frac{1}{2}\sum_{j,k=1}^m \varphi_{\bbeta_t,jk}d\beta^j_t d\beta^k_t  \\
=& ( \mathcal{L}_{\bbeta_t} v_t + \mathcal{K}_t ) dt +\epsilon \tilde{B}(t,u_t,\bbeta_t)dW_t \label{eq: mild solution for v t}
\end{align}
where
\begin{align}
\tilde{B}(s,z,\balpha) : \; & \mathbb{R}^+ \times H \times \mathbb{R}^m \to \mathcal{L}(H,H) \\
\tilde{B}(s,z,\balpha) =& B(s,z) - \sum_{j=1}^m \varphi_{\balpha , j} \mathcal{Y}_j(s,z,\balpha) \label{eq: tilde B definition} \\
\mathcal{K}_t :=& f(u_t) - f(\varphi_{\bbeta_t}) -Df(\varphi_{\bbeta_t})\cdot v_t -\sum_{i=1}^m \varphi_{\bbeta_t,i}\mathcal{V}_i(t,u_t, \bbeta_t)\nonumber \\ &- \frac{\epsilon^2}{2} \sum_{j,k,p,q=1}^m \varphi_{\bbeta_t,jk} \mathcal{N}_{jp}(u_t,\bbeta_t)\mathcal{N}_{kq}(u_t,\bbeta_t) \langle B^*(t,u_t) \psi_{\bbeta_t}^p , B^*(t,u_t) \psi^q_{\bbeta_t} \rangle ,
\end{align}
and we have substituted the expression for $d\beta^j_t d\beta^k_t$ in \eqref{eq: d beta j d beta k} . The solution for $v_t$, written in mild form, satisfies for $t\in [t_k , t_{k+1}]$,
\begin{equation}\label{eq: SDE v t}
v_t =\tilde{U}(t_k , t)v_{t_k} + \int_{t_k}^t \tilde{U}(s,t)\mathcal{K}_s ds  
+ \epsilon \int_{t_k}^t \tilde{U}(s,t) \tilde{B}(s,u_s,\bbeta_s)dW_s,
\end{equation}
where $\tilde{U}(s,t)$ is the inhomogeneous semigroup generated by $\mathcal{L}_{\bbeta_t}$. That is, for any $z\in \mathcal{D}(A)$, $\tilde{U}(s,t)\cdot z$ := $ x_t$, where
\[
\frac{dx_r}{dr} = \mathcal{L}_{\bbeta_r}\cdot x_r,
\]
and $x_s = z$. This definition can be continuously extended to all $z\in H$. Define
\begin{multline}
\xi_n := \inf \bigg\lbrace t\in[0, \tau]: \det(\mathcal{M}(u_{t},\bbeta_t)) = n^{-1}\text{ or } \norm{\int_0^t B(s,u_s)dW_s}\geq n\\ \text{ or }\sup_{1\leq i \leq m} \norm{\int_0^t \mathcal{Y}_i(s,u_s,\bbeta_s)dW_s}\geq n \text{ or }\sup_{1\leq i \leq m} \ |\beta^i_t | \geq n \bigg\rbrace.
\end{multline}
 It may be seen that $(\xi_n)_{n\geq1}$ is nondecreasing, and that $\lim_{n\to\infty}\xi_n = \tau$ a.s. Define for any $t\geq0$ $\bbeta^{n}_t = \bbeta_{t\wedge \xi_n}$ and $v^{n}_t = v_{t\wedge \xi_n}$ where as above $(v_t)_{t\geq0}$ is $u_t - \varphi_{\bbeta_t}$, and $u_t$ is the solution to the SDE in Proposition \ref{K-S}. 
Let $\Pi = (t_i)_{i=1}^M$ be a partition of $[0,t]$ for some $t\geq0$. For some family $\lbrace \theta_k\rbrace_{k=1}^{M-1} \subset [0,1]$ to be specified below, set $w_k = \theta_k v^{n}_{t_{k}} + (1-\theta_k) v^{n}_{t_{k+1}}$ and $\zeta_k = \theta_k \bbeta^{n}_{t_{k}} + (1-\theta_k) \bbeta^{n}_{t_{k+1}}$. 
Let $X_k = (v^{n}_{t_{k+1}}-v^{n}_{t_k},\bbeta^{n}_{t_{k+1}}-\bbeta^{n}_{t_k})$. 

We now write $\tilde{\mathcal{G}}_i(v_t,\bbeta_t) := \mathcal{G}_i(v_t + \varphi_{\bbeta_t},\bbeta_t)$: in this way $\tilde{\mathcal{G}}_i$ is a function on a Hilbert space, rather than the Banach space $E$, and this simplifies the calculations. From the expressions in \eqref{eq: M definition} and \eqref{eq: G second derivative}, it is clear that
\begin{align*}
\frac{\partial\tilde{\mathcal{G}}_i}{\partial \alpha^j}(z,\balpha) = &\delta(i,j) - \mathcal{M}_{ij}(z,\balpha) \\
\frac{\partial^2\tilde{\mathcal{G}}_i}{\partial \alpha^j \partial \alpha^k}(z,\balpha) = & \langle z , \psi^i_{\balpha,jk}\rangle .
\end{align*}
Thus by Taylor's theorem,
\begin{multline}
\label{taylor}
\tilde{\mathcal{G}}_i(v^{n}_t,\bbeta^{n}_t) -\tilde{\mathcal{G}}_i(v^{\bbeta}_0,\bbeta_0) = \sum_{k=1}^{M-1}\bigg\lbrace \sum_{j=1}^m \big( \delta(i,j)-\mathcal{M}_{ij}(u^n_{t^k}, \bbeta_{t^k}) \big) (\bbeta^{n,j}_{t_{k+1}} - \bbeta^{n,j}_{t_k})+ \langle v^{n}_{t_{k+1}} -v^{n}_{t_k},\psi^i_{\bbeta^{n}_{t_{k}}}\rangle  \\
 +\sum_{p=1}^m\bigg(  \frac{1}{2}\sum_{q=1}^m\frac{\partial^2 \tilde{\mathcal{G}}_i}{\partial \beta_t^p\partial \beta_t^q}\bigg|_{\zeta_k,w_k} (\beta^{n,p}_{t_{k+1}} - \beta^{n,p}_{t_k})  (\beta^{n,q}_{t_{k+1}} - \beta^{n,q}_{t_k}) + \langle v^{n}_{t_{k+1}}-v^{n}_{t_k} , \psi^i_{\zeta_k,p} \rangle (\beta^{n,p}_{t_{k+1}} - \beta^{n,p}_{t_k})  \bigg)\bigg\rbrace ,
\end{multline}
for some $\lbrace \theta_k\rbrace_{k=1}^{M-1} \subset [0,1]$. Now for any $r \geq 0$,  it must be that that
\begin{align*}
\lim_{s\to 0} \sup_{r\leq z \leq t \leq r+s}(t-z)^{-1}\langle \tilde{U}(z, t)v_{z} , \psi^i_{\bbeta^{n}_{z}}\rangle =&\lim_{s\to 0} \sup_{r\leq z \leq t \leq r+s} (t-z)^{-1}\langle v_{z} , \tilde{U}(z ,t)^* \psi^i_{\bbeta^{n}_{z}}\rangle \\ 
\to & \langle \mathcal{L}_{\bbeta_r}^*\psi^i_{\bbeta_{r}},v_r\rangle ,
\end{align*}
by the dominated convergence theorem, and using our assumption that $\psi^i_{\balpha}$ is in the domain of $ \mathcal{L}_{\balpha}^*$ for any $\balpha \in \mathbb{R}^m$. We thus find that, using \eqref{eq: mild solution for v t}, and recalling that by definition, $\mathcal{L}^*_{\bbeta_s}\psi^i_{\bbeta_s} = 0$,
\begin{align*}
\sum_{k=1}^{M-1} \langle v^{n}_{t_{k+1}}& - v^{n}_{t_k},\psi^i_{\bbeta^{n}_{t_{k}}}\rangle \\
&\to \int_{0}^{\xi_n\wedge t}[\langle \mathcal{L}_{\bbeta_s}^*\psi^i_{\bbeta_{s}},v_s\rangle + \langle \mathcal{K}_s,\psi^i_{\bbeta_{s}}\rangle ]ds +\epsilon \int_{0}^{\xi_n\wedge t}\langle\psi^i_{\bbeta_{s}},\tilde{B}(s)dW_s\rangle \\
&= \int_{0}^{\xi_n\wedge t}\langle \mathcal{K}_s,\psi^i_{\bbeta_{s}}\rangle ds +\epsilon \int_{0}^{\xi_n\wedge t}\langle\psi^i_{\bbeta_{s}},B(s,u_s)dW_s\rangle - \epsilon \big( \beta^{n,i}_{t\wedge \xi^n} - \beta^{n,i}_0 \big)
\end{align*}
since $\langle \psi^i_{\bbeta_s} , \phi_{\bbeta_s,j} \rangle = \delta(i,j)$. Clearly
\[
 \sum_{k=1}^{M-1} \sum_{j=1}^m \delta(i,j)(\beta^{n,j}_{t_{k+1}} - \beta^{n,j}_{t_k}) \to \beta^{n,i}_{t \wedge \xi^n} - \beta^{n,i}_0.
\]
Similarly, making use of \eqref{eq: beta SDE definition 2}, as $\Pi \to 0$,
\begin{equation*}
\sum_{k=1}^{M-1} \sum_{j=1}^m \mathcal{M}_{ij}(u^n_{t^k}, \bbeta_{t^k}) (\beta^{n,j}_{t_{k+1}} - \beta^{n,j}_{t_k}) \to \epsilon \int_{0}^{t\wedge\xi_n}\langle \psi^i_{\bbeta_s} , B(s,u_s)dW_s\rangle +\int_{0}^{t\wedge\xi_n} \mathcal{Z}^i_s ds,
\end{equation*}
where
\begin{multline}
\mathcal{Z}^i_t =  \epsilon^2 \sum_{j,p=1}^m \mathcal{N}_{jp}(u_t,\bbeta_t) \langle B^*(t,u_t)\psi^p_{\bbeta_t} , B^*(t,u_t) \psi^i_{\bbeta_t,j} \rangle\\
+  \frac{\epsilon^2}{2} \sum_{j,k,p,q=1}^m \langle u_t , \psi^i_{\bbeta_t,jk} \rangle \mathcal{N}_{jp}(u_t,\bbeta_t)\mathcal{N}_{kq}(u_t,\bbeta_t) \langle B^*(t,u_t) \psi_{\bbeta_t}^p , B^*(t,u_t) \psi^q_{\bbeta_t} \rangle \\
+ \langle f(u_t) - f(\varphi_{\bbeta_t}) -Df(\varphi_{\bbeta_t}) \cdot (u_t - \varphi_{\bbeta_t}), \psi^i_{\bbeta_t} \rangle .
\end{multline}
It remains for us to deal with the second order terms in the Taylor expansion \eqref{taylor}. As $\Pi \to 0$, we find that
\begin{multline}
 \frac{1}{2} \sum_{k=1}^{M-1}\sum_{p,q=1}^m\frac{\partial^2 \tilde{\mathcal{G}}_i}{\partial \beta_t^p\partial \beta_t^q}\bigg|_{\zeta_k,w_k} (\beta^{n,p}_{t_{k+1}} - \beta^{n,p}_{t_k})  (\beta^{n,q}_{t_{k+1}} - \beta^{n,q}_{t_k})\\
 \mapsto  \frac{\epsilon^2}{2} \sum_{j,k,p,q=1}^m \int_0^{t\wedge \xi_n} \langle v_s , \psi^i_{\bbeta_s,jk} \rangle   \mathcal{N}_{jp}(u_s,\bbeta_s)\mathcal{N}_{kq}(u_s,\bbeta_s) \langle B^*(s,u_s) \psi_{\bbeta_s}^p , B^*(s,u_s) \psi^q_{\bbeta_s} \rangle ds ,
\end{multline}
using the expression for the covariation in \eqref{eq: d beta j d beta k} . Note that in the above,
\begin{equation}
 \langle v_s , \psi^i_{\bbeta_s,jk} \rangle = \langle u_s,\psi^i_{\bbeta_s,jk} \rangle +  \langle \varphi_{\balpha,j} , \psi^i_{\balpha,k} \rangle,
\end{equation}
since by the integration by parts formula, $\langle \varphi_{\balpha,j} , \psi^i_{\balpha,k} \rangle = -\langle \varphi_{\balpha} , \psi^i_{\balpha,jk} \rangle$.

Finally
\begin{multline*}
\lim_{M\to\infty}\sum_{k=1}^{M-1}\sum_{p=1}^m\langle v^{n}_{t_{k+1}}-v^{n}_{t_k} , \psi^i_{\zeta_k,p} \rangle (\beta^{n,p}_{t_{k+1}} - \beta^{n,p}_{t_k}) \\
=\lim_{M\to\infty}\sum_{k=1}^{M-1} \sum_{p=1}^m\big\lbrace\langle u^{n}_{t_{k+1}}-u^{n}_{t_k}  , \psi^i_{\zeta_k,p} \rangle (\beta^{n,p}_{t_{k+1}} - \beta^{n,p}_{t_k}) \\-\sum_{j=1}^m\langle \varphi_{\bbeta^{n}_{t_k},j}(\beta^{n,j}_{t_{k+1}} - \beta^{n,j}_{t_k})  , \psi^i_{\zeta_k,p} \rangle (\beta^{n,p}_{t_{k+1}} - \beta^{n,p}_{t_k}) \big\rbrace\\
\mapsto  \epsilon^2 \int_0^{t\wedge \xi_n} \big\lbrace  \sum_{j,p=1}^m\mathcal{N}_{jp}(u_s,\bbeta_s) \langle B^*(s,u_s)\psi^p_{\bbeta_s} , B^*(s,u_s) \psi^i_{\bbeta_s,j} \rangle\\ - \sum_{j,p,q,r=1}^m\langle \varphi_{\bbeta_s,j},\psi^i_{\bbeta_s,r} \rangle   \mathcal{N}_{jp}(u_s,\bbeta_s)\mathcal{N}_{rq}(u_s,\bbeta_s) \langle B^*(s,u_s) \psi_{\bbeta_s}^p , B^*(s,u_s) \psi^q_{\bbeta_s} \rangle \big\rbrace ds.
\end{multline*}
Adding up the above identities, we find that
\[
\tilde{\mathcal{G}}_i(v_t , \bbeta_t) = 0,
\]
for all $t\leq \xi_n$. Since this holds for any $n$ and $\xi_n\uparrow \tau$ we have the result.
\end{proof}

\subsection{Isochronal Phase}\label{Section Isochronal}
In Section \ref{sec:stoch wave}, we defined the variational phase to satisfy the orthogonality relationship in \eqref{eq: G i definition}. We now outline a slightly different phase $\bgamma_t$ - the isochronal phase - which will be necessary to obtain the ergodicity result of section 6. The reason why a different phase definition is needed is that the drift of $d\bbeta_t$ is $O(\norm{v_t}^2 + \epsilon^2)$, and since for all $t\leq \tau$, $\norm{v_t}$ is typically $O(\epsilon)$, the $O(\norm{v_t}^2) $ terms make a non-trivial contribution to the drift dynamics. The benefit of using $\bgamma_t$ is that the leading order of the drift of $d\bgamma_t$ does not depend on the amplitude $v_t$. $\bgamma_t$ is an analog of the isochronal phase used in the phase reduction of finite-dimensional oscillators  \cite{Ermentrout2010,Bressloff2018}. 

We require the following additional assumptions throughout this section.
\begin{assump}\label{eq: uniform continuity P A}
For all $T \geq 0$,
\[
\lim_{h\to 0}\sup_{t\in [0,T]}h^{-1/2}\norm{ P^A(h)u_t - u_t} = 0,
\]
$\mathbb{P}$-almost-surely.
\end{assump}
In future work, in the case that $A$ is elliptic, the following assumption could likely be considerably relaxed.
\begin{assump}\label{Assumption Theta}
We assume that for some choice of orthonormal basis $\lbrace e_j \rbrace_{j\geq 1}$ for $H$,
\begin{equation}
\lim_{M\to\infty} \sup_{u \in E} \sum_{j=M}^\infty \norm{B(u)e_j}^2 = 0.
\end{equation}
\end{assump}
Write $\Phi: E \to \mathbb{R}^m$ to be the phase map of Section \ref{sec:stoch wave}, satisfying the implicit relationship
\begin{equation}
\mathcal{G}_i\big(u , \Phi(u)\big) = 0 \text{ for all }1\leq i \leq m.
\end{equation} 
Let $\mathcal{R} \subset E$ be a subset of the attracting basin of the manifold $\lbrace \varphi_{\balpha} \rbrace_{\balpha\in \mathbb{R}^m}$, with
\begin{equation}
\mathcal{R} = \big\lbrace u \in E : \norm{u - \varphi_{\Phi(u)}} \leq \delta \big\rbrace,
\end{equation}
for some $\delta$ to be determined more precisely below. We assume that $\mathcal{R}$ is sufficiently close to the manifold $\lbrace \varphi_{\balpha} \rbrace_{\balpha \in \mathbb{R}^m}$ that $\Phi$ is uniquely well-defined on $\mathcal{R}$. $\delta$ will be chosen to be small enough that if $z \in \mathcal{R}$ then there exists a unique $\boldsymbol\eta_z \in \mathbb{R}^m$ such that
\begin{align}
\lim_{t\to\infty} u_t =& \varphi_{\boldsymbol\eta_z} \text{ where }\\
u_t =& P^A(t)\cdot z + \int_0^t P^A(t-s)\cdot f(u_s) ds \text{ and }u_0 = z.\label{eq: ut dynamics}
\end{align}
Define the corresponding map to be
\begin{equation}\label{eq: Theta Definition}
\Theta: \mathcal{R} \to \mathbb{R}^m \; \; : \; \; \Theta(z) := \boldsymbol\eta_z .
\end{equation}
We require an SDE expression for $\Theta(u_t)$. To do this, we must obtain a deeper understanding of the map $\Theta$: establishing that (i) it is always well-defined for $u\in \mathcal{R}$ (as long as $\delta$ is sufficiently small), and (ii) that it is twice continuously Frechet-differentiable, which will allow us to (iii) apply Ito's Lemma to obtain an SDE expression for $d\bgamma_t$. We also wish to show that $\Theta$ is very close to the variational phase of the previous section, which will help us obtain a more tractable expression for the occupation time estimates in Section 6.
 
\begin{lemma}\label{Lemma Theta Well Defined}
There exists $\delta > 0$ such that for all $u \in E$ such that $\norm{u- \varphi_{\Phi(u)}} \leq \delta$, $\Theta(u)$ is uniquely well-defined.
\end{lemma}
\begin{proof}
The proof employs a contraction mapping and is an adaptation of \cite[Chapter 5]{Volpert1994}. Define $\mathcal{X} \subset \mathcal{C}( [0,\infty) , H) \times \mathbb{R}^m$ to be the Banach Space of all $(v,\balpha)$ such that the following norm is finite,
\begin{equation}
\norm{(v,\balpha)}_b = \norm{\balpha} + \sup_{t\geq 0} \big\lbrace \exp(bt)\norm{v_t} \big\rbrace.
\end{equation}
For some $(v,\balpha) \in \mathcal{X}$, define $\Gamma_u(v,\balpha) \in \mathcal{C}([0,\infty),H)$ to be
  \begin{align}
\Gamma_u(v,\balpha)_t =& V_{\boldsymbol\alpha}(t) \cdot (u - \varphi_{\balpha}) + \int_0^t V_{\balpha}(t-s) \cdot \lbrace f(v_s + \varphi_{\balpha}) - f(\varphi_{\balpha}) - Df(\varphi_{\balpha})\cdot v_s \big\rbrace ds\nonumber \\
 & - \sum_{i=1}^m\varphi_{\balpha,i} \bigg\langle \psi^i_{\balpha} , \int_t^\infty  \lbrace f(v_s + \varphi_{\balpha}) - f(\varphi_{\balpha}) - Df(\varphi_{\balpha})\cdot v_s \big\rbrace ds \bigg\rangle,\label{eq: hat v definition}
 \end{align}
recalling the definition of $V_{\balpha}(t)$ in \eqref{eq: norm V bound}. For $u\in E$, define $\Lambda_u : \mathbb{R}^m \to \mathbb{R}^m$ to be, for $1\leq i \leq m$,
 \begin{multline}\label{eq: Lambda u definition}
 \Lambda_u(\balpha,v)^i =\Phi(u)^i +\big\langle \psi^i_{\balpha} - \psi^i_{\Phi(u)} , u - \varphi_{\Phi(u)} \big\rangle  -\big\langle \psi^i_{\balpha} , \varphi_{\balpha} - \varphi_{\Phi(u)} - (\alpha^i -\Phi(u)^i)\varphi_{\balpha,i} \big\rangle\\ + \bigg\langle \psi^i_{\balpha} , \int_0^\infty  \lbrace f(\varphi_{\balpha} + v_s) - f(\varphi_{\balpha}) - Df(\varphi_{\balpha})\cdot v_s \big\rbrace ds \bigg\rangle
 \end{multline}
It suffices for us to show that, for $\delta $ small enough, there exists a unique $(\boldsymbol\eta , v)$ such that $\boldsymbol\eta = \Lambda_u(\boldsymbol\eta , v)$ and $v = \Gamma_u(\boldsymbol\eta,v)$. This is because, using the fact that $\big\langle \psi^i_{\Phi(u)} , u - \varphi_{\Phi(u)} \big\rangle =0$, and $\langle \psi^i_{\balpha}  , \phi_{\balpha,i} \rangle = 1$, the existence of a fixed point $(\boldsymbol\eta, \hat{v}(u) )$ implies that
 \begin{equation}\label{eq: Lambda u definition}
\varphi_{\boldsymbol\eta}^i \bigg\langle \psi^i_{\boldsymbol\eta} ,u - \varphi_{\boldsymbol\eta} + \int_0^\infty  \lbrace f(\varphi_{\boldsymbol\eta} + \hat{v}_s(u) )- f(\varphi_{\boldsymbol\eta}) - Df(\varphi_{\boldsymbol\eta})\cdot \hat{v}_s(u) \big\rbrace ds \bigg\rangle=0.
 \end{equation}
Adding this to \eqref{eq: hat v definition}, and recalling that $P_{\boldsymbol\eta} = \sum_{i=1}^m \varphi_{\boldsymbol\eta} \langle \psi^i_{\boldsymbol\eta} , \cdot \rangle$, we obtain that
  \begin{multline}
 \hat{v}_t(u) = \big\lbrace V_{\boldsymbol\eta}(t) + P_{\boldsymbol\eta} \big\rbrace \cdot (u - \varphi_{\boldsymbol\eta}) + \int_0^t V_{\boldsymbol\eta}(t-s) \cdot \lbrace f(\hat{v}_s(u) + \varphi_{\boldsymbol\eta}) - f(\varphi_{\boldsymbol\eta}) - Df(\varphi_{\boldsymbol\eta})\cdot \hat{v}_s(u) \big\rbrace ds \\
   + P_{\boldsymbol\eta}\cdot \int_0^t  \lbrace f(\hat{v}_s(u) + \varphi_{\boldsymbol\eta}) - f(\varphi_{\boldsymbol\eta}) - Df(\varphi_{\boldsymbol\eta})\cdot \hat{v}_s(u) \big\rbrace ds .\label{eq: hat v definition 2}
 \end{multline}
 Since $U_{\boldsymbol\eta}(t) = P_{\boldsymbol\eta} + V_{\boldsymbol\eta}(t)$, we obtain that
   \begin{multline}
 \hat{v}_t(u) =  U_{\boldsymbol\eta}(t)  \cdot (u - \varphi_{\boldsymbol\eta}) + \int_0^t U_{\boldsymbol\eta}(t-s) \cdot \lbrace f(\hat{v}_s(u) + \varphi_{\boldsymbol\eta}) - f(\varphi_{\boldsymbol\eta}) - Df(\varphi_{\boldsymbol\eta})\cdot \hat{v}_s(u) \big\rbrace ds 
  \end{multline}
  Writing $\hat{u}_t(u) = \varphi_{\boldsymbol\eta} +  \hat{v}_t(u) $, this means that $\hat{u}_t$ satisfies the dynamics
  \begin{equation}\label{eq: hat u dynamics}
 \hat{u}_t(u) = P^A(t)\cdot u + \int_0^t P^A(t-s)\cdot f(\hat{u}_s(u)) ds.
\end{equation}
Since $\norm{\hat{v}_t(u)} \to 0$ as $t\to\infty$, this means that $\boldsymbol\eta = \Theta(u)$.

It remains for us to show that there exists a unique $(\boldsymbol\eta,\hat{v})$ such that $\boldsymbol\eta = \Lambda_u(\boldsymbol\eta , \hat{v})$ and $\hat{v} = \Gamma_u(\boldsymbol\eta,\hat{v})$. The existence and uniqueness follows from the contraction mapping theorem. The following lemma contains the identities necessary for us to be able to apply the theorem.
\begin{lemma}\label{Lemma fixed point contraction}
For constant $\delta_2 > 0$, let $\mathcal{Z}_u(\delta_2)$ be all $(\balpha,v) \in \mathcal{X}$ such that
\begin{align}
\norm{v}_b := \sup_{t\geq 0}\big\lbrace \exp(bt)\norm{v_t}\big\rbrace &\leq \delta_2 \\
\norm{\balpha - \Phi(u)} &\leq \delta_2^{3/2}.
\end{align}
For all sufficiently small $\delta_2$, and taking $\delta$ to be sufficiently small (recall that $\mathcal{R} = \lbrace u\in E : \norm{u - \Phi(u)} \leq \delta \rbrace$), for all $u\in \mathcal{R}$,
\begin{equation}\label{eq: Theta existence inclusion}
\big(\Lambda_u(\balpha,v) , \Gamma_u(\balpha,v) \big) \in \mathcal{Z}_u(\delta_2) \text{ whenever } (\balpha,v) \in \mathcal{Z}_u(\delta_2).
\end{equation}
One can also choose $\delta$ and $\delta_2$ to be such that there exists $\rho \in (0,1)$ such that for all $(\balpha ,v), (\bbeta , z) \in \mathcal{Z}_u(\delta_2)$,
\begin{multline}
\norm{\Lambda_u(\balpha,v) - \Lambda_u(\bbeta, z)} + \sup_{t\geq 0}\big\lbrace \exp(bt)\norm{\Gamma_u(\balpha,v)_t - \Gamma_u(\bbeta,z)_t}\big\rbrace  \\ \leq \rho \norm{\balpha - \bbeta} +\rho \sup_{t\geq 0}\big\lbrace \exp(bt)\norm{v_t-z_t} \big\rbrace .\label{eq: to prove Theta contract}
\end{multline}
\end{lemma}
\begin{proof}
We start by proving \eqref{eq: Theta existence inclusion}. Let $(\balpha,v) \in \mathcal{Z}_u(\delta_2)$. Define $F: \mathbb{R}^m \times H \to H$ to be
\[
F(\balpha, v) =  f(\varphi_{\balpha} +v) - f(\varphi_{\balpha}) - Df(\varphi_{\balpha})\cdot v
\]
The Frechet differentiability of $f$ implies that there is a constant $C_F$ such that $\norm{F(\balpha,v)} \leq C_F\norm{v}^2$.  Using the bound on $\norm{V_{\balpha}(t)}$ in \eqref{eq: norm V bound},
 \begin{multline}
 \norm{\Gamma_u(v,\balpha)} \leq \mathfrak{c}\norm{u-\varphi_{\balpha}}\exp(-bt) + \mathfrak{c}C_F\int_0^t \exp\lbrace 2sb -2tb \rbrace \norm{v_s}^2 ds +C_F\sum_{i=1}^m \norm{\varphi_{\balpha,i}}\int_t^{\infty} \norm{v_s}^2 ds.
 \end{multline}
 Using the triangle inequality,
 \[
 \norm{u-\varphi_{\balpha}} \leq \norm{u - \varphi_{\Phi(u)}} + \norm{\varphi_{\Phi(u)} - \varphi_{\balpha}} \leq \delta + \norm{\balpha - \Phi(u)}\norm{D\phi(w)},
 \]
 where $w$ is in the convex hull of $\varphi_{\Phi(u)}$ and $\varphi_{\balpha}$, using Taylor's Theorem. Since, by assumption, $\sup_{\balpha \in \mathbb{R}^m, 1\leq i \leq m}\big| \varphi_{\balpha,i} \big| < \infty$, we can choose $\delta$ and $\delta_2$ to be such that
 \begin{align*}
\mathfrak{c}\big\lbrace \delta + \delta_2^{3/2} \sup_{\balpha \in \mathbb{R}^m, 1\leq i \leq m}\big| \varphi_{\balpha,i} \big| \big\rbrace \leq \delta_2 / 3,
 \end{align*}
 and this ensures that
 \[
  \mathfrak{c}\norm{u-\varphi_{\balpha}}\exp(-bt) \leq \delta_2 / 3.
 \]
 Similarly, since $(\balpha,v) \in \mathcal{Z}(\delta_2)$, 
 \[
 \mathfrak{c}C_F\int_0^t \exp\lbrace 2sb -2tb \rbrace \norm{v_s}^2 ds \leq \mathfrak{c} C_F \delta_2^2 \times t \exp(-2tb)  \leq \exp(-bt)\delta_2 / 3,
\]
for all sufficiently small $\delta_2$. Finally,
\[
C_F\sum_{i=1}^m \norm{\varphi_{\balpha,i}}\int_t^{\infty} \norm{v_s}^2 ds \leq C_F \sum_{i=1}^m \norm{\varphi_{\balpha,i}}\delta_2^2 \int_t^{\infty} \exp(-2bs)ds \leq \exp(-bt)\delta_2 /3,
\]
for small enough $\delta_2$. Taken together, the above equations imply that
\begin{equation}
 \norm{\Gamma_u(v,\balpha)_t} \leq \exp(-bt)\delta_2,
\end{equation}
once $\delta_2$ is small enough. Bounding $\big|\Lambda(\balpha,v)^i -\Phi(u)^i\big|$ analogously,  
\begin{align*}
\big|\big\langle \psi^i_{\balpha} - \psi^i_{\Phi(u)} , u - \varphi_{\Phi(u)} \big\rangle \big| &= O\big( \norm{\Phi(u) - \balpha}\norm{u - \varphi_{\Phi(u)} } \big) = O\big( \delta \delta_2^{3/2}\big) \\
\big|\big\langle \psi^i_{\balpha} , \varphi_{\balpha} - \varphi_{\Phi(u)} - (\alpha^i -\Phi(u)^i)\varphi_{\balpha,i} \big\rangle\big| &= O\big(\norm{\Phi(u) - \balpha}^2 \big) \\
\bigg| \bigg\langle \psi^i_{\balpha} , \int_0^\infty  \lbrace f(\varphi_{\balpha} + v_s) - f(\varphi_{\balpha}) - Df(\varphi_{\balpha})\cdot v_s \big\rbrace ds \bigg\rangle \bigg| &= O\big( \delta_2^2 \big).
\end{align*}
Thus for small enough $\delta_2$, it must be that $\big(\Lambda_u(\balpha,v) , \Gamma_u(\balpha,v) \big) \in \mathcal{Z}_u(\delta_t)$, as required.

It remains for us to prove \eqref{eq: to prove Theta contract}. It can be shown \cite[Lemma 1.1, Chapter 5]{Volpert1994} that there exist constants $C,C_2 > 0$ such that for all $v,w \in H$, $\norm{v},\norm{w} \leq C_2$,
\begin{align*}
\norm{ F(\balpha,v) - F(\bbeta,v) } &\leq C\norm{\balpha-\bbeta} \norm{v} \\
\norm{ F(\balpha,v) - F(\balpha,w) } &\leq C(\norm{v} + \norm{w}) \norm{v-w}.
\end{align*}
These identities allow one to straightforwardly prove \eqref{eq: to prove Theta contract}.
\end{proof}
Lemma \ref{Lemma Theta Well Defined} now follows immediately from the application of a fixed point theorem to the results in Lemma \ref{Lemma fixed point contraction}. 
\end{proof}
Now that we have defined the isochronal phase $\Theta(u)$, our next step is to establish that it is twice continuously Frechet differentiable. For any $u\in \mathcal{R}$, define the following auxiliary variables $\lbrace \hat{u}_t(u) \rbrace_{t\geq 0}$, such that 
 \begin{align}
 \hat{u}_t(u) =& P^A(t)\cdot u + \int_0^t P^A(t-s) f(  \hat{u}_s(u) )ds.\label{eq: hat u dynamics} 
 \end{align}
 $\hat{u}_t(u)$ is the solution of the deterministic dynamics, started at $u$, in the absence of noise. We know that $\lim_{t\to\infty}\hat{u}_t(u) = \varphi_{\Theta(u)}$. Define $\hat{v}_t(u) = \hat{u}_t(u) - \varphi_{\Theta(u)}$. As we demonstrated in the proof of  Lemma \ref{Lemma fixed point contraction}, $\hat{v}_t(u)$ satisfies the identity
 \begin{multline}
\hat{v}_t(u) =  -\sum_{i=1}^m\varphi_{\Theta(u),i} \bigg\langle \psi^i_{\Theta(u)} ,\int_t^{\infty}   F\big( \Theta(u),\hat{v}_s(u)\big)  ds \bigg\rangle + \\
 V_{\Theta(u)}(t) \cdot (u - \varphi_{\Theta(u)}) + \int_0^t V_{\Theta(u)}(t-s) \cdot F\big( \Theta(u),\hat{v}_s(u)\big) ds.\label{eq: hat v definition 2}
 \end{multline}
 The following lemma contains useful identities on the regularity of the isochronal phase map.
 \begin{lemma}\label{Lemma Regularity Theta}
 One can choose $\delta$ to be sufficiently small that there exists a constant $\tilde{C}$ such that for all $u \in \mathcal{R}$,
 \begin{align}
\norm{\Phi(u) - \Theta(u)} &\leq \tilde{C}\norm{u - \varphi_{\Theta(u)}}^2 \label{eq: to establish error Theta Phi}\\
 \norm{\hat{v}(u) - V_{\Theta(u)}\cdot (u - \varphi_{\Theta(u)})}_b &\leq \tilde{C}\norm{u - \varphi_{\Theta(u)}}^2. \label{eq: to establish error v V}\\
   \norm{\hat{v}(u)}_b &\leq \mathfrak{c}\norm{\hat{v}_0(u)}\label{eq: to establish bound decay hat v t}\\
   \norm{\hat{v}(u) - \hat{v}(z)}_b &\leq \tilde{C}\norm{u - z} \text{ for all }u,z\in \mathcal{R}.\label{eq: Lipschitz hat v}
 \end{align}
 \end{lemma}
 \begin{proof}
 We start by establishing \eqref{eq: to establish bound decay hat v t}. Using the triangle inequality, we find that
 \begin{equation}
 \norm{\hat{v}_t(u)} \leq  C_F\sum_{i=1}^m\norm{\varphi_{\Theta(u),i}}\norm{\psi^i_{\Theta(u)}}  \int_t^{\infty}\norm{\hat{v}_s(u)}^2 ds + \mathfrak{c}\exp(-bt) \norm{\hat{v}_0(u)} + C_F\mathfrak{c} \int_0^t \exp(bs-bt)\norm{\hat{v}_s(u)}^2 ds.
 \end{equation}
 Multiplying both sides by $\exp(bt)$ and substituting the definition of $\norm{\hat{v}(u)}_b$, we obtain that
  \begin{align}
\exp(bt) \norm{\hat{v}_t(u)} \leq  &C_F \norm{\hat{v}(u)}^2_b\sum_{i=1}^m\norm{\varphi_{\Theta(u),i}}\norm{\psi^i_{\Theta(u)}} \exp(bt) \int_t^{\infty}\exp(-2bs) ds + \mathfrak{c}\norm{\hat{v}_0(u)}\nonumber \\ &+ \mathfrak{c} C_F\norm{\hat{v}(u)}^2_b \int_0^t \exp(-bs)ds.
 \end{align}
 We thus find that there exists a constant $\hat{C}$ such that
 \begin{equation}
0 \leq \hat{C} \norm{\hat{v}(u)}^2_b- \norm{\hat{v}(u)}_b + \mathfrak{c}\norm{\hat{v}_0(u)}
 \end{equation}
Finding the roots of the quadratic on the right hand side, and assuming that \newline$\delta_2 < \big(1 + \sqrt{1-4\mathfrak{c}\hat{C}\norm{\hat{v}_0(u)}}\big) / (2\hat{C})$ (recall that by definition $\norm{\hat{v}(u)}_b \leq \delta_2$), the above equation then implies that
 \begin{equation}
 \norm{\hat{v}(u)}_b \leq  \big(1 - \sqrt{1-4\mathfrak{c}\hat{C}\norm{\hat{v}_0(u)}}\big) / (2\hat{C})  \leq \mathfrak{c}\norm{\hat{v}_0(u)},
 \end{equation}
 since the curve $x \to \sqrt{x}$ is concave at $1$, it must lie beneath its tangent line. We have thus established \eqref{eq: to establish bound decay hat v t}.
 
 We saw in the previous lemma that $\norm{F(\Theta(u),v)} \leq C_F\norm{v}^2$. It then follows easily from \eqref{eq: hat v definition 2} that
 \begin{align*}
\norm{ \hat{v}_t(u) -  V_{\Theta(u)}(t) \cdot (u - \varphi_{\Theta(u)})} &\leq \text{Const}\times \exp(-bt) \norm{\hat{v}(u)}_b^2 \\
&\leq \text{Const}\times \mathfrak{c} \exp(-bt) \norm{ u -\varphi_{\Theta(u)}}^2,
 \end{align*}
thanks to \eqref{eq: to establish bound decay hat v t}. This establishes \eqref{eq: to establish error v V}. 
 
We next establish \eqref{eq: to establish error Theta Phi}. Recall that $(\Theta(u),\hat{v}(u))$ is a fixed point of \eqref{eq: Lambda u definition}, which means that $\Theta(u) = \Lambda_u\big( \Theta(u) , \hat{v}(u) \big)$. \eqref{eq: Lambda u definition} implies that
 \begin{multline}\label{eq: leading order Theta u}
 \norm{\Theta(u) - \Phi(u) -  \big\langle \psi^i_{\balpha} , \int_0^\infty  \lbrace f(\varphi_{\balpha} + \hat{v}_s) - f(\varphi_{\balpha}) - Df(\varphi_{\balpha})\cdot \hat{v}_s \big\rbrace ds \big\rangle} \\= O\big( \norm{\Theta(u)-\Phi(u)}\norm{u - \varphi_{\Phi(u)}} +  \norm{\Theta(u)-\Phi(u)}^2\big)
 \end{multline}
Furthermore, using \eqref{eq: to establish error v V},
\begin{align*}
 f(\varphi_{\balpha} + \hat{v}_s) - f(\varphi_{\balpha}) - Df(\varphi_{\balpha})\cdot \hat{v}_s &= D^{(2)}f(\varphi_{\balpha})\cdot \hat{v}_s \cdot \hat{v}_s + O\big( \norm{\hat{v}_s(u)}^3 \big) \\
 &=  D^{(2)}f(\varphi_{\balpha})\cdot (V_{\balpha}(s) \cdot \hat{v}_0(u)) \cdot (V_{\balpha}(s)\cdot \hat{v}_0(u) )+ O\big( \norm{u - \varphi_{\Theta(u)}}^3 \big)
\end{align*}
The proof of \eqref{eq: Lipschitz hat v} is omitted.
 
 \end{proof}

  \begin{lemma}\label{Lemma Theta Derivative Approximation}
(i) $\Theta: \mathcal{R} \to \mathbb{R}^m$ is twice continuously Frechet differentiable for all $u$ in the interior of $\mathcal{R}$. The first and second Frechet derivatives at $u \in \mathcal{R}$ are written as $D\Theta(u): H \to H$ and $D^{(2)}\Theta(u): H \times H \to H$ (it must be emphasized that we are defining $D\Theta(u)$ and $D^{(2)}\Theta(u)$ to act on $H$, not $E$). 

(ii) The function
\[
u \to \sum_{j=1}^\infty \big\lbrace D^{(2)}\Theta (u) \cdot B(u)e_j \cdot B(u)e_j \big\rbrace,
\]
where $\lbrace e_j \rbrace_{j\geq 1}$ is the orthonormal basis for $H$ of Assumption \ref{Assumption Theta}, is continuous over $\mathcal{R}$.

(iii) Furthermore there exists a constant $C$ such that for all $u\in \mathcal{R}$, and all $w\in H$, writing $\balpha = \Theta(u)$, for $1\leq i \leq m$,
\begin{align}
\big| D\Theta^i(u)\cdot w - \big\langle \psi^i_{\balpha} , w \big\rangle \big| \leq C \norm{u - \varphi_{\balpha}}\norm{w}&\label{eq: Theta approximation 1} \\
\| D^{(2)}\Theta^i(u) \cdot w \cdot w -2\sum_{j=1}^m \langle \psi^j_{\balpha} ,  w \rangle  \langle \psi^i_{\balpha,j} ,  w \rangle +\sum_{j,k=1}^m \big\langle \varphi_{\balpha,j} , \psi^i_{\balpha,k} &\big\rangle \langle \psi^j_{\balpha} ,w \rangle \langle \psi^k_{\balpha} ,w \rangle\nonumber\\ - \int_0^{\infty}\big\langle \psi^i_{\balpha} , D^{(2)}f(\varphi_{\balpha})\cdot V_{\balpha}(s) w \cdot V_{\balpha}(s) w \big\rangle ds  \| &\leq C \norm{u - \varphi_{\balpha}}\norm{w}^2.\label{eq: Theta approximation 2}
\end{align}
\end{lemma}
\begin{proof}
The Frechet Differentiability of $\Theta$ is established by `implicitly differentiating' the fixed point identities \eqref{eq: hat u dynamics} and  $\Theta(u) = \Lambda_u\big( \Theta(u) , \hat{v}(u) \big)$ that define $(\Theta(u),\hat{v}_u)$. That is, fix $z\in H$ and define $w_n = u + n^{-1}z \in \mathcal{R}$, for large enough $n > 0$. We then find that there exists
$(D\Theta(u)\cdot z , \delta \hat{v}(u)\cdot z) \in \mathcal{X}$ such that
\begin{align}
\lim_{n\to\infty} \norm{ n \delta\hat{v}(w_n) - n\hat{v}(u) - \delta\hat{v}(u)\cdot z  }_b =& 0 \text{ where }\\
\delta\hat{v}_t(u) =& \delta\hat{u}_t(u) - \sum_{j=1}^m \varphi_{\Theta(u),j}\big(D\Theta(u)\cdot z \big)^j \\
\text{For all }t\geq 0 \; \; , \; \; \delta\hat{u}_t(u)\cdot z =& P^A(t)\cdot z + \int_0^t P^A(t-s)\cdot Df(\hat{u}_s(u))\cdot \big( \delta\hat{u}_s(u)\cdot z\big) ds\nonumber\\
D\Theta(u)\cdot z =& D\Lambda_u\big(\Theta(u) , \hat{v}(u)\big) \cdot ( D\Theta(u)\cdot z , \delta\hat{v}(u)\cdot z ),
\end{align}
and in this last expression, $D\Lambda_u: \mathcal{X} \to \mathbb{R}^m$ denotes the Frechet Derivative. The second Frechet Derivative, in directions $w,z \in H$, is established analogously, as follows. We find that there exists $\delta^{(2)} \hat{v}(u)\cdot w \cdot z \in \mathcal{C}([0,\infty), H)$ such that $\norm{\delta^{(2)} \hat{v}(u)\cdot w \cdot z }_b < \infty$ and $D^{(2)}\Theta(u)\cdot w \cdot z \in \mathbb{R}^m$ such that $\norm{ \delta^{(2)} \hat{v}(u)\cdot w \cdot z }_b < \infty$ and for all $t\geq 0$,
\begin{align}
\delta^{(2)}\hat{v}_t(u)\cdot w \cdot z =&  \int_0^t P^A(t-s)\cdot \big\lbrace Df(\hat{u}_s(u))\cdot \big(\delta^{(2)}\hat{v}_s(u)\cdot w \cdot z\big) \\&+D^{(2)}f(\hat{u}_s(u))\cdot \big(\delta\hat{v}_s(u)\cdot w\big) \cdot \big(\delta\hat{v}_s(u) \cdot z\big)\big\rbrace ds \\
D^{(2)}\Theta(u)\cdot w \cdot z =& D^{(2)}\Lambda_u\big(\Theta(u) , \hat{v}(u)\big) \cdot ( D\Theta(u)\cdot z , \delta\hat{v}(u)\cdot z ) \cdot ( D\Theta(u)\cdot w , \delta\hat{v}(u)\cdot w )\nonumber \\&+D\Lambda_u\big(\Theta(u) , \hat{v}(u)\big) \cdot ( D^{(2)}\Theta(u)\cdot w \cdot z , \delta^{(2)}\hat{v}(u)\cdot w \cdot z )
\end{align}
(ii) follows from Assumption \ref{Assumption Theta}.

The proof of \eqref{eq: Theta approximation 1} and \eqref{eq: Theta approximation 2} proceeds from the approximation of $\Theta$ in \eqref{eq: leading order Theta u}. This result means that $D\Phi(u)\cdot w = D\Theta(u)\cdot w$, to leading order. Applying the implicit function theorem to the identities $ \mathcal{G}_j(u,\Phi(u)) = 0$ (for $1\leq j \leq m$) implies that for each $1\leq i \leq m$ and $w\in H$,
\begin{align}
\big(D\Phi(u)\cdot w \big)^i =& \sum_{j=1}^m \mathcal{N}_{ij}(u,\balpha) \langle w,\psi^j_{\balpha} \rangle \\
\big(D^{(2)}\Phi(u)\cdot w \cdot w\big)^i  =&  \sum_{r=1}^m \mathcal{N}_{ir}(u,\balpha) \bigg\lbrace  2\sum_{j,p=1}^m \mathcal{N}_{jp}(u,\balpha) \langle \psi^p_{\balpha} , w\rangle \langle \psi^r_{\balpha,j},w \rangle\nonumber\\
+   &\sum_{j,k,p,q=1}^m \langle u , \psi^r_{\balpha,jk} \rangle \mathcal{N}_{jp}(u,\balpha)\mathcal{N}_{kq}(u,\balpha) \langle  \psi_{\balpha}^p , w\rangle \langle  \psi^q_{\balpha},w \rangle \bigg\rbrace .
\end{align}
However $\mathcal{N}_{ij}(u,\Phi(u)) = \delta(i,j) + O(\| u - \varphi_{\Theta(u)}\|)$, and 
\[
\langle u , \psi^r_{\balpha,jk}\rangle = \langle \varphi_{\balpha}, \psi^r_{\balpha,jk} \rangle + O(\| u - \varphi_{\Theta(u)}\|)=- \langle \varphi_{\balpha,j}, \psi^r_{\balpha,k} \rangle + O(\| u - \varphi_{\Theta(u)}\|),
\]
using integration by parts, and we thus find that
\begin{align}
(D\Phi(\varphi_{\balpha})\cdot w)^i =& \big\langle \psi^i_{\balpha} , w \big\rangle +O\big(\norm{u - \varphi_{\balpha}}\norm{w}\big) \\
\big(D^{(2)}\Phi( \varphi_{\balpha})\cdot w \cdot w\big)^i =& 2\sum_{j=1}^m \langle \psi^j_{\balpha} ,  w \rangle  \langle \psi^i_{\balpha,j} ,  w \rangle -\sum_{j,k=1}^m \big\langle \varphi_{\balpha,j} , \psi^i_{\balpha,k} \big\rangle \langle \psi^j_{\balpha} ,w \rangle \langle \psi^k_{\balpha} ,w \rangle\nonumber \\ &+ O\big(\norm{u - \varphi_{\balpha}}\norm{w}^2\big)
\end{align}
Finally, we apply the approximation in \eqref{eq: leading order Theta u}, and use the approximation \newline$ \norm{\hat{v}(u) - V_{\Theta(u)}\cdot (u - \varphi_{\Theta(u)})}_b \leq \tilde{C}\norm{u - \varphi_{\Theta(u)}}^2$ from Lemma \ref{Lemma Regularity Theta}.
\end{proof}
We are now ready to define a stochastic process for the isochronal phase. For $t < \tau_i := \inf\lbrace s \geq 0: u_s \notin \mathcal{R} \rbrace$, define the $\mathbb{R}^m$-valued stochastic process $\bgamma_t$ to be
\begin{equation}\label{eq: existence gamma}
\bgamma_t = \epsilon^2 \int_0^t \sum_{j=1}^\infty \big\lbrace D^{(2)}\Theta (u_s) \cdot B(u_s)e_j \cdot B(u_s)e_j \big\rbrace ds + \epsilon \int_0^t D\Theta(u_s) \cdot B(u_s)dW_s,
\end{equation}
where $\lbrace e_j \rbrace_{j=1}^\infty$ is the orthonormal basis of $H$ of Assumption \ref{Assumption Theta}. Since the coefficient functions in the above integral are continuous, $\bgamma_t$ is well-defined.

\begin{lemma}\label{Lemma gamma SDE}
For all $t < \tau_i$, $\bgamma_t = \Theta(u_t)$.
\end{lemma}
\begin{proof} 
The proof is essentially a generalization of Ito's Lemma to our infinite dimensional problem: it is similar to the proof of Lemma \ref{Lemma Beta SDE}. For some $\hat{\epsilon} > 0$, define  
\[
\xi^n = \inf\big\lbrace t\leq \tau_i - \hat{\epsilon} : \norm{u_t - \varphi_{\bgamma_t}} = n \text{ or }\sup_{1\leq i \leq m} |\gamma^i_t | = n\big\rbrace
\]
For any $T > 0$, we discretize $[0,T \wedge \xi_n]$ into a partition $\lbrace t_i \rbrace_{i=1}^M$. Using a second-order Taylor expansion (which is possible because $\Theta$ is twice-Frechet-differentiable, as noted in \eqref{Lemma Theta Derivative Approximation}),
\begin{equation}
\Theta(u_t) - \Theta(u_0) = \sum_{i=1}^M \big\lbrace D\Theta(u_{t_i})\cdot ( u_{t_{i+1}} - u_{t_i}) + D^{(2)}\Theta(\tilde{u}_{i})\cdot ( u_{t_{i+1}} - u_{t_i}) \cdot ( u_{t_{i+1}} - u_{t_i}) \big\rbrace
\end{equation}
where $\tilde{u}_i = \lambda_i u_{t_i} + (1-\lambda_i) u_{t_{i+1}}$ for some $\lambda_i \in [0,1]$. Now
\begin{equation}\label{eq: in Theta proof u t i decomposition}
u_{t_{i+1}} - u_{t_i} = \lbrace P^A(t_{i+1} - t_i) - I \rbrace \cdot u_{t_i} + \int_{t_i}^{t_{i+1}} P^A(t_{i+1} - s)f(u_s) ds  +\epsilon \int_{t_i}^{t_{i+1}} P^A(t_{i+1} - s) B(u_s) dW_s
\end{equation}
As the partition $\Delta \to 0$, 
\begin{align}
  \sum_{i=1}^M D\Theta(u_{t_i}) \int_{t_i}^{t_{i+1}} P^A(t_{i+1} - s) B(u_s) dW_s &\to \int_0^{T\wedge \xi_n} D\Theta(u_s)\cdot B(u_s) dW_s.
\end{align}
We now show that as $\Delta \to 0$,
\begin{equation}\label{eq: to show first derivative zero}
\sum_{i=1}^M D\Theta\big(\hat{u}_{t_i}  \big)\cdot\bigg\lbrace  \lbrace P^A(t_{i+1} - t_i) - I \rbrace \cdot u_{t_i} + \int_{t_i}^{t_{i+1}} P^A(t_{i+1} - s)f(u_s) ds \bigg\rbrace \to 0.
\end{equation}
For $t\in [t_i , t_{i+1}]$, define $\hat{u}_t$ to satisfy the deterministic flow with initial condition $u_{t_i}$, i.e. $\hat{u}_{t_i} = u_{t_i}$ and
\begin{equation} \label{eq: deterministic flow semi group}
\hat{u}_t = P^A(t - t_i)u_{t_i} + \int_{t_i}^{t} P^A(t - s) f(\hat{u}_s) ds.
\end{equation}
Now for all $t\in [t_i,t_{i+1}]$, $\Theta(\hat{u}_t) = \Theta(u_{t_i})$, because the isochronal phase is (by definition) invariant under the deterministic flow. Since $\Theta$ is Frechet differentiable, by Taylor's Theorem there must exist $\hat{\lambda}_i \in [0,1]$ such that, writing $\hat{\hat{u}}_i =\hat{\lambda}_i \hat{u}_{t_i} + (1-\hat{\lambda}_i) \hat{u}_{t_{i+1}} $,
\begin{equation}
D\Theta\big(\hat{\hat{u}}_i  \big)\cdot (\hat{u}_{t_{i+1}} - \hat{u}_{t_i}) = 0.
\end{equation}
Applying a second Taylor expansion to the above identity, one obtains that
\begin{equation}
D\Theta\big(\hat{u}_{t_i}  \big)\cdot (\hat{u}_{t_{i+1}} - \hat{u}_{t_i}) + D^{(2)}\Theta\big( \tilde{\hat{u}}_i  \big)\cdot (\hat{u}_{t_{i+1}} - \hat{u}_{t_i}) \cdot (\hat{\hat{u}}_i - \hat{u}_{t_i}) = 0,
\end{equation}
and $\tilde{\hat{u}}_i $ is in the convex hull of $\hat{u}_i$ and $\hat{\hat{u}}_i$. Now it follows straightforwardly from the uniformly Lipschitz property of $f$ that
\begin{multline}
\bigg\| \hat{u}_{t_i+1} - \hat{u}_{t_i} -  \lbrace P^A(t_{i+1} - t_i) - I \rbrace \cdot u_{t_i} + \int_{t_i}^{t_{i+1}} P^A(t_{i+1} - s)f(u_s) ds \bigg\|\\ \leq \text{Const} \times (t_{i+1} - t_i) \times \sup_{s\in [t_i, t_{i+1}]} | u_s - \hat{u}_s(u_{t_i})|.
\end{multline}
One easily checks using Gronwall's Inequality (and the Lipschitz nature of $f$) that
\[
|u_s - \hat{u}_s(u_{t_i})| = O\big( (t_{i+1} - t_i) \big).
\]
We thus find that
\begin{multline}
D\Theta\big(\hat{u}_{t_i}  \big)\cdot  \bigg( P^A(t_{i+1} - t_i) - I \rbrace \cdot u_{t_i} + \int_{t_i}^{t_{i+1}} P^A(t_{i+1} - s)f(u_s) ds\bigg) =\\ - D^{(2)}\Theta\big( \tilde{\hat{u}}_i  \big)\cdot (\hat{u}_{t_{i+1}} - \hat{u}_{t_i}) \cdot (\hat{\hat{u}}_i - \hat{u}_{t_i})  + O \big( (t_{i+1} - t_i)^2 \big).
\end{multline}

Now, making use of the expression in \eqref{eq: deterministic flow semi group}, since $\sup_{t \in [0,T]}\| P^A(t) \| < \infty$, and $\norm{f(u_s)}$ is uniformly bounded,
\begin{align*}
\norm{\hat{u}_t - \hat{u}_{t_i}} \leq &\norm{P^A(t - t_i)u_{t_i} - u_{t_i}} + O( t - t_i) \\
=& o\big( (t_{i+1}- t_i)^{1/2}\big),
\end{align*}
using Assumption \ref{eq: uniform continuity P A}. Similarly $ (\hat{\hat{u}}_i - \hat{u}_{t_i}) =  o\big( (t_{i+1}- t_i)^{1/2}\big)$. We thus find that
\[
(t_{i+1} - t_i)^{-1} D^{(2)}\Theta\big( \tilde{\hat{u}}_i  \big)\cdot (\hat{u}_{t_{i+1}} - \hat{u}_{t_i}) \cdot (\hat{\hat{u}}_i - \hat{u}_{t_i})  \to 0,
\]
as the partition goes to zero.  This means that \eqref{eq: to show first derivative zero} must hold, as required.

It remains to show that as $\Delta \to 0$, for any $1\leq p \leq m$,
\begin{equation}\label{eq: to show final Theta}
\bigg| \sum_{i=1}^M D^{(2)}\Theta^p(\tilde{u}_{i})\cdot ( u_{t_{i+1}} - u_{t_i}) \cdot ( u_{t_{i+1}} - u_{t_i}) -\epsilon^2 \int_0^{T \wedge \xi_n}  \sum_{j=1}^\infty \big\lbrace D^{(2)}\Theta^p (u_s) \cdot B(u_s)e_j \cdot B(u_s)e_j \big\rbrace ds \bigg| \to 0 .
\end{equation}
Now, substituting the decomposition of $u_t$ in \eqref{eq: in Theta proof u t i decomposition}, and writing
\begin{equation}\label{eq: in Theta proof u t i decomposition}
Z_i = \lbrace P^A(t_{i+1} - t_i) - I \rbrace \cdot u_{t_i} + \int_{t_i}^{t_{i+1}} P^A(t_{i+1} - s)f(u_s) ds,  
\end{equation}
we obtain that
\begin{multline}
\sum_{i=1}^M D^{(2)}\Theta^p(\tilde{u}_{i})\cdot ( u_{t_{i+1}} - u_{t_i}) \cdot ( u_{t_{i+1}} - u_{t_i}) = \sum_{i=1}^M\bigg\lbrace D^{(2)}\Theta^p(\tilde{u}_{i})\cdot Z_i \cdot Z_i \\+ 2\epsilon D^{(2)}\Theta^p(\tilde{u}_{i})\cdot Z_i \cdot \int_{t_i}^{t_{i+1}} P^A(t_{i+1} - s) B(u_s) dW_s \\ + \epsilon^2 D^{(2)}\Theta^p(\tilde{u}_{i})\cdot  \int_{t_i}^{t_{i+1}} P^A(t_{i+1} - s) B(u_s) dW_s\cdot \int_{t_i}^{t_{i+1}} P^A(t_{i+1} - s) B(u_s) dW_s \bigg\rbrace.
\end{multline}
One easily shows that, as $\Delta \to 0$, and making use of Assumption \ref{eq: uniform continuity P A},
\[
 \sum_{i=1}^M\bigg\lbrace D^{(2)}\Theta^p(\tilde{u}_{i})\cdot Z_i \cdot Z_i \\+ 2\epsilon D^{(2)}\Theta^p(\tilde{u}_{i})\cdot Z_i \cdot \int_{t_i}^{t_{i+1}} P^A(t_{i+1} - s) B(u_s) dW_s\bigg\rbrace \to 0.
\]
Let $\lbrace w^j_t \rbrace_{j\geq 1}$ be independent Brownian motions. We can represent $W_t$ as $\sum_{j=1}^\infty w^j_t e_j$, where $\lbrace e_j \rbrace_{j\geq 1}$ is the orthonormal basis for $H$ of Assumption \ref{Assumption Theta}, noting that the summation does not converge in $H$ (but only in an appropriate ambient Hilbert space \cite{DaPrato2014}). We fix an integer $K > 0$ and find that
\begin{multline}
 \int_{t_i \wedge \xi_n}^{t_{i+1}\wedge \xi_n}  P^A(t_{i+1} - s)  B(u_s)\cdot dW_s  = \sum_{j=K+1}^\infty\int_{t_i \wedge \xi_n}^{t_{i+1}\wedge \xi_n} P^A(t_{i+1} - s)   B(u_s) e_j dw^j_s + X_i \text{ where }\nonumber
\end{multline}
$ X_i = \sum_{j=1}^K\int_{t_i \wedge \xi_n}^{t_{i+1}\wedge \xi_n}   P^A(t_{i+1} - s) B(u_s) e_j dw^j_s$. It thus follows from Assumption \ref{Assumption Theta} that for each $n \in \mathbb{Z}^+$ there must exist $K_n$ such that for all $K \geq K_n$,
\begin{multline}
 \mathbb{E} \bigg[\bigg| D^{(2)}\Theta^p(\tilde{u}_{i})\cdot  \int_{t_i}^{t_{i+1}} P^A(t_{i+1} - s) B(u_s) dW_s\cdot \int_{t_i}^{t_{i+1}} P^A(t_{i+1} - s) B(u_s) dW_s\\-  D^{(2)}\Theta^p(\tilde{u}_{i})\cdot  \int_{t_i}^{t_{i+1}} P^A(t_{i+1} - s) B(u_s) dW_s\cdot \int_{t_i}^{t_{i+1}} P^A(t_{i+1} - s) B(u_s) dW_s \bigg|^2 \bigg] \leq n^{-1} (t_{i+1}-t_i).
\end{multline}
However the continuity of $D^{(2)}\Theta_p$ implies that
\begin{multline*}
 \bigg| \sum_{i=1}^M D^{(2)}\Theta^p(\tilde{u}_{i})\cdot Z_i \cdot Z_i - \epsilon^2\int_0^{T\wedge\xi_n} \sum_{j=1}^{K_n} \big( D^{(2)}\Theta^p (u_t) \cdot B(u_t)e_j \cdot B(u_t)e_j \big) dt \bigg| \to 0,
\end{multline*}
as $\Delta \to 0$.
\end{proof}

\section{Long-Time Stability}
\label{Section Long Time Stability}

The main aim of this section is to show that the probability of the stochastic system leaving a close neighborhood of the manifold parameterized by $\lbrace \varphi_{\balpha} \rbrace_{\balpha \in  \mathbb{R}^m}$ after an exponentially long period of time, is exponentially unlikely. This result is necessary for the metastable results of the next section, and is also of independent interest. The main result of this section is Theorem \ref{Theorem exponential stability time}.

Define the `amplitude' of the solution (relative to the nearest shifted pattern / wave) to be
\begin{equation}
v_t = u_t - \varphi_{\bbeta_t},
\end{equation}
recalling the definition of the variational phase $\bbeta_t$ in \eqref{eq: G i definition}. Let $\bar{\kappa}$ be a positive constant such that $\bar{\kappa} \leq \frac{1}{2m} \sup_{1\leq i,j \leq m}\| \psi^i_{0,j}\|^{-1}$. We require an additional condition on $\bar{\kappa}$ further on, in \eqref{eq: bar kappa second}. Define the stopping time, for some $\kappa \leq \bar{\kappa} := \frac{1}{2m} \sup_{1\leq i,j \leq m}\| \psi^i_{0,j} \| ^{-1}$,
\begin{equation}
\eta = \inf \big\lbrace t\geq 0\; : \; \norm{v_t} = \kappa \big\rbrace .
\end{equation}
This constant has been chosen such that, if $\norm{v_t} \leq \bar{\kappa}$, then necessarily
\begin{equation}\label{eq: norm M lower bound}
 \sum_{i,j=1}^m \mathfrak{a}_i \mathfrak{a}_j\mathcal{M}_{ij}(u_t,\bbeta_t) - \frac{1}{2} \sum_{i=1}^m \mathfrak{a}_i^2 \geq 0,
\end{equation}
for all $\mathfrak{a} \in \mathbb{R}^m$. The above identity can be inferred from the definition in \eqref{eq: M definition}. This implies that (i) the matrix $\mathcal{M}(u_t,\bbeta_t)$ can be inverted if $t\leq \eta$, and (ii) the map $(u_t,\bbeta_t) \to \mathcal{N}(u_t,\bbeta_t)$ is locally Lipschitz.

The main result of this section is the following.
\begin{theorem}\label{Theorem exponential stability time}
Let $p \in [0,2)$. There exists a constant $C > 0$ (independent of the choice of $p$) and $\epsilon_{(p)} > 0$ such that for all $\epsilon \in (0, \epsilon_{(p)})$, and all $\kappa\in [\epsilon_{(p)}^p , \bar{\kappa}]$ and all $T > 0$,
\begin{equation}
\mathbb{P}\big( \sup_{t\in [0,T]} \norm{v_t} > \kappa \big) \leq T \exp\big(-C \epsilon^{-2} \kappa^2 \big)
\end{equation}
\end{theorem}
If one desires an optimal value for the constant $C$, one would need to perform a more detailed Large Deviations analysis (see for instance \cite{Salins2019}), and this would likely require extensive computation\footnote{In general, research on the existence of Large Deviation principles is vastly more developed than research on the efficient numerical computation of Large Deviations rate functions / first exit times.}. Even so, to the best of this author's knowledge, the bound in Theorem \ref{Theorem exponential stability time} is the most optimal one in the literature for this sort of problem.
 
We discretize time into intervals of length $\Delta t := b^{-1}\log(4 \mathfrak{c}^{-1})$, and we write $t_a := a \Delta t$. Write $u_a := u_{t_a}$ and $\bbeta_a := \bbeta_{t_a}$. Define the event
\begin{align}\label{eq: A a definition}
\mathcal{A}_a =  \big\lbrace  \norm{v_a} \leq \kappa / (2\mathfrak{c})  \big\rbrace \cap \big\lbrace \norm{v_{a+1}} > \kappa / (2\mathfrak{c}) \text{ or } \sup_{t\in [t_a, t_{a+1}]}\norm{v_t} > \kappa  \big\rbrace .
\end{align}
Write $R = \lfloor T / \Delta t \rfloor$. Noting that $\norm{v_0} = 0$, a union-of-events bound implies that
\begin{equation}\label{eq: union of events bound}
\mathbb{P}\big( \sup_{t\in [0, R\Delta t]} \norm{v_t} > \kappa \big) \leq \sum_{a=0}^R \mathbb{P}\big( \mathcal{A}_a \big).
\end{equation}
 It follows from Ito's Lemma that
\begin{align}
dv_t =& \big(Au_t + f(u_t)\big) dt +\epsilon B(t,u_t)dW_t -\sum_{i=1}^m \varphi_{\bbeta_t , i}d\bbeta^i_t - \frac{1}{2}\sum_{j,k=1}^m \varphi_{\bbeta_t,jk}d\beta^j_t d\beta^k_t \\
=& \big( \mathcal{L}_{\bbeta_a}v_t + f(u_t) - f(\varphi_{\bbeta_a})-Df(\varphi_{\bbeta_a})\cdot v_t\big) dt +\epsilon B(t,u_t)dW_t -\sum_{i=1}^m \varphi_{\bbeta_t , i}d\beta^i_t\nonumber \\  &- \frac{1}{2}\sum_{j,k=1}^m \varphi_{\bbeta_t,jk}d\beta^j_t d\beta^k_t ,
\end{align}
and we have substituted the identity $A\varphi_{\bbeta_a} + f(\varphi_{\bbeta_a}) = 0$. Write $U_{a}(t) := U_{\bbeta_a}(t)$ to be the semigroup generated by $\mathcal{L}_{\bbeta_a}$. Taking the mild solution, and substituting the expression for $d\beta^j_t d\beta^k_t$ in \eqref{eq: d beta j d beta k} , we have that
\begin{equation}\label{eq: SDE v t}
v_t = U_a(t-t_a)v_a + \int_{t_a}^t U_a(t-s)\mathcal{H}_s ds  
+ \epsilon \int_{t_a}^t U_a(t-s) \tilde{B}(s,u_s,\bbeta_s)dW_s,
\end{equation}
where
\begin{align}
\tilde{B}(s,z,\balpha) : \; & \mathbb{R}^+ \times H \times \mathbb{R}^m \to \mathcal{L}(H,H) \\
\tilde{B}(s,z,\balpha) =& B(s,z) - \sum_{j=1}^m \varphi_{\balpha , j} \mathcal{Y}_j(s,z,\balpha) \label{eq: tilde B definition} \\
\mathcal{H}_s :=& f(u_s) - f(\varphi_{\bbeta_a})-Df(\varphi_{\bbeta_a})\cdot v_s -\sum_{i=1}^m \varphi_{\bbeta_s,i}\mathcal{V}_i(s,u_s, \bbeta_s) \\ &- \frac{\epsilon^2}{2} \sum_{j,k,p,q=1}^m \varphi_{\bbeta_s,jk} \mathcal{N}_{jp}(u_s,\bbeta_s)\mathcal{N}_{kq}(u_s,\bbeta_s) \langle B^*(s,u_s) \psi_{\bbeta_s}^p , B^*(s,u_s) \psi^q_{\bbeta_s} \rangle .\nonumber
\end{align}

\begin{lemma}\label{eq: bound covariation terms}
There exist constants $C_1 , C_2, C_3$ such that for all $t \in [t_a, t_{a+1} \wedge \eta]$, where $t_a = \sup\lbrace t_b : t_b \leq t \rbrace$,
\begin{multline}
\int_{t_a}^t U_a(t-s)\mathcal{H}_s ds  \leq \epsilon^2 C_1 + C_3\sup_{s\in [t_a,t]}\norm{v_s}^2 \\
+ \epsilon C_2  \sup_{s\in [t_a, \eta \wedge t_{a+1}]} \sup_{1\leq i \leq m}  \bigg| \int_{t_a}^{s} \mathcal{Y}_i(r,u_r,\bbeta_r) dW_r \bigg| 
\end{multline}
\end{lemma}
\begin{proof}
Using the definition of the semigroup in \eqref{eq: V definition},
\begin{align*}
\norm{U_{a}(t-s)}_{\mathcal{L}} &\leq \norm{V_{\bbeta_a}(t-s)}_{\mathcal{L}} + \norm{P_{\bbeta_a}}_{\mathcal{L}} \\
& \leq \mathfrak{c} + \sum_{i=1}^m \norm{\varphi_{\bbeta_a,i}}\| \psi^i_{\bbeta_a}\| 
=  \mathfrak{c} + \sum_{i=1}^m \norm{\varphi_{0,i}}\| \psi^i_{0}\| 
\end{align*}
using \eqref{eq: norm V bound}, and the fact that $P_{\bbeta_a}$ is an $m$-dimensional projection. Let $\bar{C}$ be a constant such that
\begin{align*}
\sup_{t\geq 0, x\in E,\balpha \in \mathbb{R}^m \; : \norm{x-\varphi_{\balpha} } \leq \bar{\kappa} , \mathcal{G}(x,\balpha) = 0}\sup_{1\leq i \leq m} \big|\mathcal{V}_i(t,x,\balpha) \big| &\leq \bar{C} \\
\sup_{t\geq 0,x\in E,\balpha \in \mathbb{R}^m \; : \norm{x-\varphi_{\balpha} } \leq \bar{\kappa} , \mathcal{G}(x,\balpha) = 0}\sup_{1\leq i \leq m} \norm{\mathcal{Y}_i(t,x,\balpha)}_{HS} &\leq \bar{C} .
\end{align*}
The constant $\bar{C}$ exists, because as noted in \eqref{eq: norm M lower bound}, the choice of $\bar{\kappa}$ ensures that the lowest eigenvalue of $\mathcal{M}(x,\balpha)$ is greater than or equal to a half.

Employing the triangle inequality,
\begin{multline*}
\norm{f(u_s) - f(\varphi_{\bbeta_a})-Df(\varphi_{\bbeta_a})\cdot v_s} \leq \norm{f(u_s) - f(\varphi_{\bbeta_s})-Df(\varphi_{\bbeta_s})\cdot v_s}\\
+\norm{Df(\varphi_{\bbeta_a})\cdot v_s - Df(\varphi_{\bbeta_s})\cdot v_s + f(\varphi_{\bbeta_s})-f(\varphi_{\bbeta_a})}
\end{multline*}
Now using the second order Taylor expansion, there exists $\lambda_a \in [0,1]$ such that, writing $\bar{u}_a = \lambda_a \varphi_{\bbeta_a} + (1-\lambda_a) u_a$,
\begin{equation}
f(u_a) - f(\varphi_{\bbeta_a})-Df(\varphi_{\bbeta_a})\cdot v_a = D^{(2)} f( \bar{u}_a)\cdot v_a \cdot v_a .
\end{equation}
The assumed boundedness of the second derivative thus implies that for some constant $C > 0$,
\begin{equation}\label{eq: second taylor f}
\norm{f(u_a) - f(\varphi_{\bbeta_a})-Df(\varphi_{\bbeta_a})v_a } \leq C\norm{v_a}^2.
\end{equation}
The boundedness of the first and second Frechet derivatives of $f$ implies that there exists a constant $C$ such that 
\begin{align}
\norm{Df(\varphi_{\bbeta_a})\cdot v_s - Df(\varphi_{\bbeta_s})\cdot v_s + f(\varphi_{\bbeta_s})-f(\varphi_{\bbeta_a})} \leq &C\big\lbrace 1+ \sup_{s\in [t_a , \eta \wedge t_{a+1}]} \norm{v_s} \big\rbrace \nonumber \\  &\times \sup_{s\in [t_a, \eta \wedge t_{a+1}]} \norm{\bbeta_s - \bbeta_a} \\
\leq &C(1+\bar{\kappa}) \sup_{s\in [t_a, \eta \wedge t_{a+1}]} \norm{\bbeta_s - \bbeta_a} ,
\end{align}
since by definition of $\eta$, $ \sup_{s\in [t_a, \eta \wedge t_{a+1}]} \norm{v_s} \leq \bar{\kappa}$. Using our SDE for $\bbeta_t$ in \eqref{eq: beta SDE definition 2}, we find that as long as $ \sup_{s\in [t_a, t_{a+1}]} \norm{v_s} \leq \bar{\kappa}$, there exists a constant $\bar{C}$ such that
\begin{multline}
\sup_{s\in [t_a, t_{a+1}], 1\leq i \leq m} \big|\bbeta^i_{s} - \bbeta^i_a \big|  \leq \Delta t \bar{C} \epsilon^2 +\epsilon \sup_{s\in [t_a, \eta \wedge t_{a+1}]} \sup_{1\leq i \leq m}  \bigg| \int_{t_a}^{s} \mathcal{Y}_i(r,u_r,\bbeta_r) dW_r \bigg| \\
+ \Delta t C \norm{\psi^i_{0}}\norm{v_t}^2
\end{multline}
and we have used the Cauchy-Schwarz Inequality to find that
\begin{align*}
\big| \langle f(u_t) - f(\varphi_{\bbeta_t}) -Df(\varphi_{\bbeta_t})\cdot (u_t - \varphi_{\bbeta_t}), \psi^i_{\bbeta_t} \rangle \big| &\leq \norm{ f(u_t) - f(\varphi_{\bbeta_t}) -Df(\varphi_{\bbeta_t})(u_t - \varphi_{\bbeta_t})}  \| \psi^i_{\bbeta_t} \| \\
&\leq  C\norm{\psi^i_{0}}\norm{v_t}^2,
\end{align*}
thanks to \eqref{eq: second taylor f}, and since $  \| \psi^i_{\bbeta_t} \| = \| \psi^i_{0} \|$.
\end{proof}
Now we insist that $\epsilon_{(p)}$ (defined in the statement of the theorem) is such that
\begin{equation}\label{eq: epsilon p definition}
0 < \epsilon_{(p)}^2 C_1 \leq \epsilon^p_{(p)}/(16\mathfrak{c}),
\end{equation}
which is always possible since by assumption $p < 2$. We also insist that $\bar{\kappa}$ is such that
\begin{equation}\label{eq: bar kappa second}
C_3\bar{\kappa}^2 \leq \bar{\kappa} / (16\mathfrak{c}),
\end{equation}
and since $\kappa \leq \bar{\kappa}$, is must be that $C_3\kappa^2 \leq \frac{\kappa}{16\mathfrak{c}}$.

\begin{lemma}
\begin{align}
\mathcal{A}_a \subseteq &   \big\lbrace  \norm{v_a} \leq \kappa / (2\mathfrak{c})  \big\rbrace \cap \lbrace \mathcal{B}_a \cup \mathcal{C}_a \rbrace\text{ where } \\
\mathcal{B}_a =&  \bigg\lbrace\epsilon C_2  \sup_{s\in [t_a, \eta \wedge t_{a+1}]} \sup_{1\leq i \leq m}  \bigg| \int_{t_a}^{s \wedge \eta} \mathcal{Y}_i(r,u_r,\bbeta_r) dW_r \bigg| \geq \frac{\kappa}{16\mathfrak{c}} \bigg\rbrace  \\
\mathcal{C}_a =& \bigg\lbrace \epsilon \sup_{t\in t_{a+1}\wedge \eta} \norm{ \int_{t_a}^{t \wedge \eta} U_a(t-s) \tilde{B}(s,u_s,\bbeta_s)dW_s } \geq \frac{\kappa}{16\mathfrak{c}} \bigg\rbrace .
\end{align}
\end{lemma}
\begin{proof}
Using the expression for $v_t$ in \eqref{eq: SDE v t} and the triangle inequality, for all $t\in [t_a , \eta \wedge t_{a+1}]$,
\begin{align}
\norm{v_t} \leq & \norm{U_a(t-t_a)v_a} + \norm{\int_{t_a}^t U_a(t-s)\mathcal{H}_s ds  }
+ \epsilon \norm{\int_{t_a}^t U_a(t-s) \tilde{B}(s,u_s,\bbeta_s)dW_s} \nonumber \\
\leq & \mathfrak{c}\exp(-b (t-t_a)) \norm{v_a}+  \epsilon^2 C_1 + \epsilon C_2  \sup_{s\in [t_a, \eta \wedge t_{a+1}]} \sup_{1\leq i \leq m}  \bigg| \int_{t_a}^{s} \mathcal{Y}_i(r,u_r,\bbeta_r) dW_r \bigg|\nonumber \\
&+C_3 \kappa^2 + \epsilon \norm{\int_{t_a}^t U_a(t-s) \tilde{B}(s,u_s,\bbeta_s)dW_s}
\end{align}
after substituting the semigroup bound in \eqref{eq: norm V bound} and the bound in \eqref{eq: bound covariation terms}. Since $t-t_a \leq  b^{-1}\log(4 \mathfrak{c})$ and $\norm{v_a} \leq \kappa / (2\mathfrak{c})$, $\mathfrak{c}\exp(-b(t-t_a))\norm{v_a} \leq \kappa / (8\mathfrak{c})$. Also, $C_3\kappa^2 \leq \kappa / (16\mathfrak{c})$. Furthermore, since by definition $\kappa > \epsilon_{(p)}^p$, the inequality in \eqref{eq: epsilon p definition} implies that $\epsilon^2 C_1 \leq \kappa / (16\mathfrak{c})$. We thus find that 
\begin{align*}
\norm{v_{a+1}} \leq \frac{\kappa}{4\mathfrak{c}} +  \epsilon C_2  \sup_{s\in [t_a, \eta \wedge t_{a+1}]} \sup_{1\leq i \leq m}  \bigg| \int_{t_a}^{s} \mathcal{Y}_i(r,u_r,\bbeta_r) dW_r \bigg|\nonumber \\
+ \epsilon \norm{\int_{t_a}^{t_{a+1}} U_a( \Delta t) \tilde{B}(s,u_s,\bbeta_s)dW_s}.
\end{align*}
Thus if $\norm{v_{a+1}} > \frac{\kappa}{2\mathfrak{c}}$, then it must be that either $\mathcal{B}_a$ or $\mathcal{C}_a$ must hold. 

The remaining event in the definition of $\mathcal{A}_a$ is $ \sup_{t\in [t_a, t_{a+1}]}\norm{v_t} > \kappa $. Notice that, since $ \norm{U_a(t-t_a)v_a} \leq \mathfrak{c}\norm{v_a} \leq \mathfrak{c}\times \kappa / (2\mathfrak{c})$,
\begin{align*}
\sup_{t\in [t_a, t_{a+1}\wedge \eta]}\norm{v_{t}} \leq  \kappa / 2 + \epsilon^2 C_1 +  \epsilon C_2  \sup_{s\in [t_a, \eta \wedge t_{a+1}]} \sup_{1\leq i \leq m}  \bigg| \int_{t_a}^{s} \mathcal{Y}_i(r,u_r,\bbeta_r) dW_r \bigg|\nonumber \\
+C_3 \kappa^2+ \epsilon \sup_{t\in [t_a, t_{a+1} ]} \norm{\int_{t_a}^{t\wedge \eta } U_a(t-s) \tilde{B}(s,u_s,\bbeta_s)dW_s} \\
\leq \kappa / 2 + \kappa / 8 +  \epsilon C_2  \sup_{s\in [t_a, \eta \wedge t_{a+1}]} \sup_{1\leq i \leq m}  \bigg| \int_{t_a}^{s} \mathcal{Y}_i(r,u_r,\bbeta_r) dW_r \bigg|\nonumber \\
+ \epsilon \sup_{t\in [t_a, t_{a+1} ]}\norm{\int_{t_a}^{t\wedge \eta} U_a(t-s) \tilde{B}(s,u_s,\bbeta_s)dW_s},
\end{align*}
since $\epsilon^2 C_1 \leq \kappa / (16\mathfrak{c})$, $C_3\kappa^2 \leq \kappa / (16\mathfrak{c})$ and $\mathfrak{c} \geq 1$. We again see that if $\sup_{t\in [t_a, t_{a+1}\wedge \eta]}\norm{v_{t}} \geq \kappa$, then $\mathcal{B}_a$ or $\mathcal{C}_a$ must hold.
\end{proof}
It thus follows from \eqref{eq: union of events bound} that
\begin{equation}
\mathbb{P}\big( \sup_{t\in [0,T]}\norm{v_t} \geq \kappa \big) \leq \sum_{a=0}^{R} \big\lbrace \mathbb{P}( \mathcal{B}_a) + \mathbb{P}( \mathcal{C}_a) \big\rbrace.
\end{equation}
Recalling that $R = \lfloor T / \Delta t\rfloor$, it then follows from Lemma \ref{lemma exponential bounds} that for a constant $C_4 > 0$,
\[
\mathbb{P}\big( \sup_{t\in [0,T]}\norm{v_t} \geq \kappa \big) \leq TC_4 \exp\big( - C\epsilon^{-2} \kappa^2 \big).
\]
For small enough $\epsilon$, this implies Theorem \ref{Theorem exponential stability time}.
\begin{lemma}\label{lemma exponential bounds}
There exists a constant $C > 0$ such that
\begin{align}
\sup_{a\geq 0}\mathbb{P}\big( \mathcal{B}_a \big) \leq \exp\big( - C\epsilon^{-2} \kappa^2 \big) \\
\sup_{a\geq 0}\mathbb{P}\big( \mathcal{C}_a \big) \leq \exp\big( - C\epsilon^{-2} \kappa^2 \big) \label{eq: bound C a}.
\end{align}
\end{lemma}
\begin{proof}
We prove the first result only. The bound of \eqref{eq: bound C a} can be obtained by taking an exponential moment of the stochastic integral: see \cite[Section 5]{MacLaurin2020a}. 
Using the definition of $\tilde{B}$ in \eqref{eq: tilde B definition},
\begin{multline}
\mathbb{P}\big( \sup_{t\in [t_a, t_{a+1} ]}\norm{\int_{t_a}^{t\wedge \eta} U_a(t-s) \tilde{B}(s,u_s,\bbeta_s)dW_s} \geq \frac{\kappa}{16\epsilon\mathfrak{c}}\big) \leq \\ \mathbb{P}\big(  \sup_{t\in [t_a, t_{a+1} ]}\norm{\int_{t_a}^{t\wedge \eta} U_a(t-s) B(s,u_s)dW_s} \geq \frac{\kappa}{32\epsilon\mathfrak{c}}\big)\\
+\mathbb{P}\big(  \sup_{t\in [t_a, t_{a+1} ]}\norm{\sum_{i=1}^m\int_{t_a}^{t\wedge \eta} U_a(t-s)\varphi_{\bbeta_s,i} \mathcal{Y}_i(s,u_s,\bbeta_s)dW_s} \geq \frac{\kappa}{32\epsilon\mathfrak{c}}\big).
\end{multline}
The bound of the second term can be obtained by taking an exponential moment: see \cite[Section 5]{MacLaurin2020a}. It follows from \cite[Theorem 1.3]{Brzezniak1999} that there exists a constant $\tilde{C}$ such that
\begin{equation}
 \mathbb{P}\big(  \sup_{t\in [t_a, t_{a+1} ]}\norm{\int_{t_a}^{t\wedge \eta} U_a(t-s) B(s,u_s)dW_s} \geq \frac{\kappa}{32\epsilon\mathfrak{c}}\big) \leq \exp\bigg( -\tilde{C}\bigg\lbrace \frac{\kappa}{32\epsilon\mathfrak{c}}\bigg\rbrace^2 \bigg).
\end{equation}
We must also make use of Assumption \ref{assump exponential moment}.
%
\end{proof}

\section{Long-Time Ergodicity of the Induced Phase Dynamics} \label{Section Ergodicity}

In this section we search for an approximate expression for the average occupation times of the phase as it wanders across the manifold $\lbrace \varphi_{\balpha} \rbrace_{\balpha\in\mathbb{R}^m}$ over very long periods of time. It is assumed throughout this section that the manifold is periodic; more precisely, that $\varphi_{\balpha}$, $\psi^i_{\balpha}$ and the first and second derivatives are $2\pi$-periodic. Write $\mathcal{S} = (\mathbb{S}^1)^m$ - where $\mathbb{S}^1$ is the ring $[-\pi , \pi]$, with the points $-\pi$ and $\pi$ identified. This means that $\lbrace \varphi_{\balpha} \rbrace_{\balpha \in \mathcal{S}}$ and $\lbrace \psi_{\balpha} \rbrace_{\balpha \in \mathcal{S}}$ constitute manifolds that are everywhere twice-continuously differentiable. Since $\mathcal{S}$ is compact, over long periods of time one expects the phase (taken modulo $\mathcal{S}$) to continually return to any particular region (as long as the noise coefficients are not degenerate). We can thus use our phase SDE to determine an approximate expression for the expected proportion of time that the phase spends in the neighborhood of any particular subset of the manifold over long periods of time. We can also determine the average shift in the phase over very long periods of time: in effect, we reach an understanding of the average rotation induced in the phase through the interaction of the manifold geometry and correlation structure of the noise. 

In earlier work we demonstrated that the wandering of a `neural bump' induced by noise and a weak non-noisy external stimulus approaches a Von Mises distribution \cite{MacLaurin2020a}. In that work, the weak stimulus dominated the $O(\epsilon^2)$ covariation terms. By contrast, in this work there is no weak stimulus, and the long-time distribution of the phase is determined by the balance between the $O(\epsilon^2)$ covariation terms in the drift, and the $O(\epsilon)$ stochastic noise.

We assume throughout this section that $B$ is independent of $t$, so that we write $B(u_t)$\footnote{In the case of spiral waves in compact domains, even if one is working in a rotating reference frame, if the driving white noise $W_t$ is cylindrical, then one does not need to co-rotate $B$. }. We also require Assumptions \ref{eq: uniform continuity P A} and \ref{Assumption Theta} in order that the isochronal phase SDE in \eqref{eq: existence gamma} is well-defined. Recall that $\lbrace e_j \rbrace_{j\geq 1}$ is the orthonormal basis for H of Assumption \ref{Assumption Theta}. Define, for $\balpha \in \mathcal{S}$, functions $\tilde{\mathcal{V}}: \mathbb{R}^m \to \mathbb{R}^m$ and  $\tilde{\mathcal{Y}}: \mathbb{R}^m \to \mathcal{L}(H , \mathbb{R}^m)$ that are (respectively) leading order approximations of $D^{(2)}\Theta(\varphi_{\balpha})$ and $D^{(1)}\Theta(\varphi_{\balpha})$. Writing $\tilde{\mathcal{V}} = (\tilde{\mathcal{V}}_i)_{1\leq i \leq m}$ and $\tilde{\mathcal{Y}} = (\tilde{\mathcal{Y}}_i)_{1\leq i \leq m}$, with $\tilde{\mathcal{Y}}_i: \mathbb{R}^m \to \mathcal{L}(H,\mathbb{R})$,
\begin{align}
\tilde{\mathcal{V}}_i(\balpha) = & \sum_{j=1}^m \langle B^*(\varphi_{\balpha})\psi^j_{\balpha} , B^*(\varphi_{\balpha}) \psi^i_{\balpha,j} \rangle
- \frac{1}{2} \sum_{j,k=1}^m \langle \varphi_{\balpha,j} , \psi^i_{\balpha,k} \rangle \langle B^*(\varphi_{\balpha}) \psi_{\balpha}^j , B^*(\varphi_{\balpha}) \psi^k_{\balpha} \rangle \nonumber \\
&+ \Gamma_{QV}(\balpha) \text{ where }\\
\Gamma_{QV}(\balpha) = & \frac{1}{2}\sum_{i=1}^m\sum_{j=1}^\infty \int_0^\infty \big\langle \psi^i_{\balpha} , D^{(2)} f(\varphi_{\balpha}) \cdot \big( V_{\balpha}(s)B(\varphi_{\balpha}) e_j \big) \cdot \big( V_{\balpha}(s) B(\varphi_{\balpha}) e_j \big) \big\rangle ds \\
\tilde{\mathcal{Y}}_i(\balpha)\cdot z = &\langle B(\varphi_{\balpha}) z,\psi^i_{\balpha} \rangle \text{ and define }\\
\mathcal{H}_{jk}(\balpha) = & \langle B^*(\varphi_{\balpha}) \psi_{\balpha}^j , B^*(\varphi_{\balpha}) \psi^k_{\balpha} \rangle .
\end{align}
It follows from Assumption \ref{Assumption Theta}, and the fact that $\norm{V_{\balpha}(t)} \leq \mathfrak{c} \exp(-bt)$ that $\Gamma_{QV}(\balpha)$ is well-defined for all $\balpha \in \mathbb{R}^m$. For any $\boldsymbol\xi \in \mathcal{S}$, define $p_t(\balpha | \boldsymbol\xi)$ to satisfy the Fokker-Planck PDE, for $\balpha \in \mathcal{S}$,
\begin{align}
\frac{\partial }{\partial t}p_t(\balpha | \boldsymbol\xi) = - \sum_{j=1}^m \frac{\partial}{\partial \alpha^j} \big\lbrace p_t(\balpha  | \boldsymbol\xi) \tilde{\mathcal{V}}_j(\balpha) \big\rbrace + \frac{1}{2}\sum_{j,k=1}^m \frac{\partial^2}{\partial \alpha^j \partial \alpha^k} \big\lbrace  \mathcal{H}_{jk}(\balpha)p_t(\balpha  | \boldsymbol\xi ) \big\rbrace , \label{eq: p alpha xi}
\end{align}
such that $\lim_{t\to 0}p_t(\balpha | \xi) = \delta_{\xi}(\balpha)$, and with periodic boundary conditions. The solution of this PDE defines a transition probability density for an $\mathcal{S}$-valued stochastic process started at $\boldsymbol\xi$. We assume that 
\begin{equation}
\inf_{\balpha \in \mathcal{S}} \det\big( \mathcal{H}(\balpha) \big) > 0,
\end{equation}
and since $\tilde{\mathcal{V}}$ and $\mathcal{H}$ are continuously differentiable, this means that the stochastic process with Fokker-Plank equation given by \eqref{eq: p alpha xi} has a unique invariant density $p_*(\balpha)$. That is, $p_*(\balpha)$ is the unique solution of
\begin{equation}
 - \sum_{j=1}^m \frac{\partial}{\partial \alpha^j} \big\lbrace p_*(\balpha  ) \tilde{\mathcal{V}}_j(\balpha) \big\rbrace + \frac{1}{2}\sum_{j,k=1}^m \frac{\partial^2}{\partial \alpha^j \partial \alpha^k} \big\lbrace  \mathcal{H}_{jk}(\balpha)p_*(\balpha ) \big\rbrace = 0,
\end{equation}
and such that $\int_{\mathcal{S}} p_*(\balpha) d\balpha = 1$. Let $P_*$ be the probability measure on $\mathcal{S}$ with density $p_*$. Standard theory \cite{Hairer2011} dictates that
\begin{equation}\label{eq: convergence pt p star}
\lim_{t \to \infty}\sup_{\boldsymbol\xi , \balpha \in \mathcal{S}} \big| p_t(\balpha | \boldsymbol\xi) - p_*(\balpha) \big|   = 0.
\end{equation}
The main result of this section is the following. Its implication is that the wandering across the manifold $\mathcal{S}$ of the phase SDE, over long periods of time, is indicated by the density $p_*$, with very high probability. For any $\balpha \in \mathbb{R}^m$, write $\balpha \mod \mathcal{S}$ to be the unique member $\boldsymbol\eta \in \mathcal{S}$ such that $\eta^i = \alpha^i + 2\pi p_i$, $\eta \in (-\pi , \pi]$ for integers $p_i$. Let $\mathcal{C}_{b,\mathcal{S}}^2(\mathbb{R}^m)$ be the set of all periodic twice continuously differentiable functions $g$ on $\mathbb{R}^m$, such that
\begin{align}
g(\balpha) = g(\balpha \mod \mathcal{S}) \\
|g| \leq 1 \; \; , \; \; \big|\frac{\partial g}{\partial \alpha^j} \big| \leq 1 \; \; , \; \;  \big|\frac{\partial g}{\partial \alpha^j \partial \alpha^k} \big| \leq 1 
\end{align}
for all $1\leq j,k \leq m$, and $\balpha \mapsto \frac{\partial g}{\partial \alpha^j \partial \alpha^k}(\balpha)$ and $\balpha \mapsto \frac{\partial g}{\partial \alpha^j}(\balpha) $ are continuous on $\mathcal{S}$. This assumption means that, thanks to Taylor's Theorem, for any $\balpha , \bbeta \in \mathcal{S}$,
\begin{equation}\label{eq: g lipschitz}
\big| g(\balpha) - g(\bbeta) \big| \leq m \norm{\balpha - \bbeta}.
\end{equation}

The significance of the first result in the following theorem is that the distribution of the stochastic phase modulo $\mathcal{S}$ over very long periods of time converges to $P_*$ as $\epsilon \to 0$. The second result determines the average asymptotic shift in the phase over very long periods of time. It parallels analogous results for the long-time average phase-shift of stochastic oscillators \cite{Giacomin2018,MacLaurin2020b}.
\begin{theorem}\label{first ergodic theorem}
For any $\delta > 0$ there exist constants $C_{\delta}  > 0$ and $\epsilon_{\delta} > 0$ such that for all $\epsilon < \epsilon_{\delta}$ and  all $g\in \mathcal{C}^2_{b,\mathcal{S}}(\mathbb{R}^m)$, writing $T_{\delta} = \exp\big( C_{\delta} \epsilon^{-2} \big)$,
\begin{align}\label{eq: first result last theorem}
\mathbb{P}\bigg( \bigg| T_{\delta}^{-1}\int_0^{T_{\delta}} g(\bgamma_s) ds-\mathbb{E}^{P_*}[g] \bigg| > \delta \bigg) &\leq \exp\big( - C_{\delta}\epsilon^{-2} \big).
\end{align}
Also,
\begin{equation}\label{eq: second result last theorem}
\mathbb{P}\big( \sup_{1\leq i \leq m}\big| \epsilon^{-2} T_{\delta}^{-1} \gamma^i_{T_{\delta}}-\mathbb{E}^{P_*}[ \tilde{\mathcal{V}}_i] \big| > \delta \big) \leq \exp\big( - C_{\delta}\epsilon^{-2}  \big) 
\end{equation}
\end{theorem}

\begin{rem}
Effectively, \eqref{eq: second result last theorem} implies that over long time scales the isochronal phase $\gamma^i_t$ changes at average rate $\epsilon^2 \mathbb{E}^{P_*}[ \tilde{\mathcal{V}}_i]$: that is, if $\mathbb{E}^{P_*}[ \tilde{\mathcal{V}}_i]$ is not equal to zero, there is a small average oscillation in the phase (taken modulo $\mathcal{S}$) induced by the noise correlation. In general, this oscillation may only be discernible over timescales of $O(\epsilon^{-2})$. Often, if the noise correlation structure satisfies certain symmetries with respect to the manifold, $\mathbb{E}^{P_*}[ \tilde{\mathcal{V}}_i]$ is identically zero. 
\end{rem}

\subsection{Proof of Theorem \ref{first ergodic theorem}}

To demonstrate the lemma, we discretize time into blocks of $\Delta t$ (to be specified more precisely further below, this definition of $\Delta t$ is different from the previous section). Write $t_a := a \Delta t$. To facilitate the proofs, we wish to define a stochastic process $\bupsilon_t^{\balpha}$ that does not depend on $\epsilon$: $\upsilon^{\balpha}_t$ will have the useful property that under the rescaling of time $t \to \epsilon^{-2} t$, its dynamics closely approximates $\bgamma_t$. In more detail, for any $\balpha \in \mathcal{S}$, define $\bupsilon_{t}^{\balpha}$ to be the solution to the stochastic process
\begin{equation}
d\bupsilon^{\balpha}_t =  \tilde{\mathcal{V}}(\bupsilon^{\balpha}_t)dt +  \tilde{\mathcal{Y}}(\bupsilon^{\balpha}_t)d\tilde{W}_t,
\end{equation}
with initial condition $\bupsilon^{\balpha}_0 = \balpha$. Here $\tilde{W}_t$ is any cylindrical Wiener process with identical probability law to $W_t$.

Define
\begin{equation}
\hat{\tau} = \inf\big\lbrace t \geq 0 \; : \norm{u_t - \varphi_{\bbeta_t}}= \zeta \big\rbrace ,
\end{equation}
for some $\zeta$ so be specified more precisely below. We assume that $\zeta \in (0 , \bar{\kappa}]$, where $\bar{\kappa}$ is the constant in Theorem \ref{Theorem exponential stability time}. As long as $\zeta$ is sufficiently small, and since  (as proved in \eqref{eq: to establish error Theta Phi}),
\begin{equation}
\norm{u_t - \varphi_{\bgamma_t}} = \norm{u_t - \varphi_{\bbeta_t}} + O\big( \norm{u_t - \varphi_{\bbeta_t}}^2\big)
\end{equation}
it must hold that
\begin{equation}
\hat{\tau}  \leq \tau_i,
\end{equation}
and the SDE for $\bgamma_t$ is well-defined for all $t\leq \tau$.
\begin{lemma}\label{Lemma T constant}
For any $\tilde{\epsilon}, \delta  > 0$, there exists $T > 0$ such that for all $t \geq T$,
\begin{equation}
\sup_{\balpha \in \mathcal{S}}\mathbb{P}\big( \big| t^{-1} \int_0^t  g\big( \bupsilon^{\balpha}_s\big) ds - \mathbb{E}^{P_*}[g] \big| \geq \delta / 2 \big) \leq \tilde{\epsilon}
\end{equation}
\end{lemma}
\begin{proof}
Write $\bar{g}(\theta) = g(\theta)- \mathbb{E}^{P_*}[g]$. By Chebyshev's Inequality,
\begin{align*}
\mathbb{P}\big( \big| t^{-1} \int_0^t  \bar{g}\big( \bupsilon^{\balpha}_s\big) ds \big| \geq \delta / 2 \big) \leq &
4\mathbb{E}\big[  \big|  \int_0^t  \bar{g}\big( \bupsilon^{\balpha}_s\big) ds \big|^2  \big] / (t\delta )^2 \\
=& \frac{4}{t^2\delta^2}  \int_0^t \int_0^t \mathbb{E}\big[ \bar{g}\big( \bupsilon^{\balpha}_r\big) \bar{g}\big( \bupsilon^{\balpha}_s\big)  \big] ds dr \\
\leq & \frac{4}{t^2\delta^2}  \int_0^t\int_0^t \chi\lbrace |r-s| \leq T_0 \rbrace \mathbb{E}\big[ \bar{g}\big( \bupsilon^{\balpha}_r\big) \bar{g}\big( \bupsilon^{\balpha}_s\big)  \big] ds dr  \\
&+ \frac{4}{t^2\delta^2}  \int_0^t \int_0^t \chi\lbrace |r-s| > T_0 \rbrace \mathbb{E}\big[ \bar{g}\big( \bupsilon^{\balpha}_r\big) \bar{g}\big( \bupsilon^{\balpha}_s\big)  \big] ds dr ,
\end{align*}
for any $T_0 > 0$. Thanks to the assumption in \eqref{eq: convergence pt p star}, we can choose $T_0$ to be such that
\begin{equation}
4 \delta^{-2}\sup_{0\leq s < t \; : |t-s| \geq T_0}\big\lbrace \mathbb{E}\big[ \bar{g}\big( \bupsilon^{\balpha}_t\big) \bar{g}\big( \bupsilon^{\balpha}_s\big)  \big] \big\rbrace \leq \tilde{\epsilon} / 2.
\end{equation}
Also
\begin{equation}
 \frac{4}{t^2\delta^2}  \int_0^t \int_0^t \chi\lbrace |r-s| \leq T_0 \rbrace \mathbb{E}\big[ \bar{g}\big( \bupsilon^{\balpha}_r\big) \bar{g}\big( \bupsilon^{\balpha}_s\big)  \big] ds dr = O\big(t^{-1}\big),
\end{equation}
since $|\bar{g}|$ is uniformly bounded. The previous three equations imply the lemma.

\end{proof}

For any $a \in \mathbb{Z}^+$, and assuming that $\tau > t_a$, define $\tilde{\bgamma}_t$ to satisfy the $\mathbb{R}^m$-valued SDE for all $t\in [t_a , t_{a+1})$,
\begin{equation} \label{eq: tilde beta t SDE}
d\tilde{\bgamma}_t = \epsilon^2 \tilde{\mathcal{V}}(\tilde{\bgamma}_t)dt + \epsilon \tilde{\mathcal{Y}}(\tilde{\bgamma}_t)dW_t,
\end{equation}
and with initial condition such that $\tilde{\bgamma}_{t_a} = \bgamma_{t_a}$. Notice that $\tilde{\bgamma}_t$ is driven by the same Brownian motion $W_t$ as $\bgamma_t$. One can easily check that a unique solution to the above SDE exists.  Furthermore $\bupsilon^{\bgamma_a}_{t}$ has the same probability law as $\tilde{\bgamma}_{t_a+\epsilon^{-2} (t-t_a)}$ for $t \in [t_a , t_{a}+T ]$. This follows from the fact that they are both Markovian Processes, with identical infinitesimal generators.

We now fix some arbitrary $\tilde{\epsilon} < 0$, and let $T$ be the constant of Lemma \ref{Lemma T constant}. Write $\Delta t = \epsilon^{-2}T$. First observe that
\begin{multline}
\mathbb{P}\big(\big|T_{\delta}^{-1}\int_0^{T_{\delta}} g(\bgamma_s) ds-\mathbb{E}^{P_*}[g] \big| \geq \delta \big)  \leq \mathbb{P}\big( \hat{\tau} < \exp(C_{\delta}\epsilon^{-2}) \big) \\
+\mathbb{P}\big(\big|T_{\delta}^{-1}\int_0^{T_{\delta}} g(\bgamma_s) ds-\mathbb{E}^{P_*}[g] \big| \geq \delta  \text{ and }\hat{\tau} \geq T_{\delta} \big). \label{eq: to show small}
\end{multline}
Using Lemma \ref{Theorem exponential stability time}, as long as $C_{\delta}$ is small enough,
\begin{equation}
\mathbb{P}\big( \hat{\tau} < \exp(C_{\delta}\epsilon^{-2}) \big) \leq \frac{1}{2}\exp(-C_{\delta}\epsilon^{-2}).
\end{equation}
It remains to show that
\begin{equation} \label{eq: to establish C delta epsilon minus 2}
\mathbb{P}\big(\big|T_{\delta}^{-1}\int_0^{T_{\delta}} g(\bgamma_s) ds-\mathbb{E}^{P_*}[g] \big| \geq \delta  \text{ and }\hat{\tau} \geq T_{\delta} \big) \leq \frac{1}{2}\exp(C_{\delta}\epsilon^{-2}) \big).
\end{equation}
By Chebyshev's Inequality, for some $\kappa > 0$,
\begin{multline}
\mathbb{P}\big(\big|T_{\delta}^{-1}\int_0^{T_{\delta}} g(\bgamma_s) ds-\mathbb{E}^{P_*}[g] \big| \geq \delta  \text{ and }\hat{\tau} \geq T_{\delta} \big) \\
\leq \mathbb{E}\bigg[\chi \lbrace \hat{\tau} \geq T_{\delta}\rbrace \bigg\lbrace \exp\bigg(\kappa \int_0^{T_{\delta}} g(\bgamma_s) ds-\kappa \hat{\tau} \mathbb{E}^{P_*}[g] -\kappa \hat{\tau} \delta  \bigg) \\+\exp\bigg(-\kappa\int_0^{T_{\delta}} g(\bgamma_s) ds+\kappa T_{\delta}\mathbb{E}^{P_*}[g] -\kappa  T_{\delta}\delta  \bigg) \bigg\rbrace\bigg]. \label{eq: end chebysehev}
\end{multline}
We next show that, for a particular choice of $\kappa$,
\begin{equation}
\mathbb{E}\bigg[\chi \lbrace \hat{\tau} \geq T_{\delta}\rbrace  \exp\bigg(\kappa \int_0^{T_{\delta}} g(\bgamma_s) ds-\kappa T_{\delta} \mathbb{E}^{P_*}[g] -\kappa T_{\delta} \delta  \bigg) \bigg] \leq \exp\big( -\text{Const}\times \epsilon^{-2} \big)
\end{equation}
and omit the analogous proof that
\begin{equation}
\mathbb{E}\bigg[\exp\bigg(-\kappa\int_0^{T_{\delta}} g(\bgamma_s) ds+\kappa T_{\delta}\mathbb{E}^{P_*}[g] -\kappa  T_{\delta} \delta  \bigg) \bigg] \leq \exp\big( -\text{Const}\times \epsilon^{-2} \big).
\end{equation}
Now define the random variable 
\[
H_a =\exp\bigg\lbrace \kappa  \big\lbrace \hat{\tau} \geq t_{a+1} \big\rbrace  \bigg( \int_{t_a}^{t_{a+1}} g(\bgamma_s) ds -  \Delta t \mathbb{E}^{P_*}[g]\bigg)  \bigg\rbrace,
\]
and define $\mathfrak{a}$ to be the random index such that $\hat{\tau} \in [t_{\mathfrak{a}}, t_{\mathfrak{a}+1})$. We thus find that
\begin{align*}
\mathbb{E}\big[\chi \lbrace \hat{\tau} \geq T_{\delta}\rbrace &\exp\big(\kappa \int_0^{T_{\delta}} g(\bgamma_s) ds-\kappa \hat{\tau} \mathbb{E}^{P_*}[g]  \big) \big]\\
&= \mathbb{E}\big[ \chi \lbrace \hat{\tau} \geq  T_{\delta}\rbrace  \prod_{a=0}^{\mathfrak{a}} H_a \times \exp\big(\kappa \int_{t_a}^{\hat{\tau}} g(\bgamma_s) ds-\kappa (\hat{\tau} -t_{\mathfrak{a}})\mathbb{E}^{P_*}[g]  \big) \big] \\
&\leq \mathbb{E}\big[ \chi \lbrace \hat{\tau} \geq T_{\delta}\rbrace  \prod_{a=0}^{\mathfrak{a}} H_a \big] \exp\big( 2 \kappa \Delta t \big),
\end{align*}
since $|g| \leq 1$ uniformly.  Furthermore 
\begin{multline*}
\mathbb{P}\big( \big| (\Delta t)^{-1} \int_{t_a}^{t_{a+1}} g(\bgamma_s) ds -  \mathbb{E}^{P_*}[g] \big| \geq \delta \text{ and } \hat{\tau} \geq t_{a+1}  \big) \\
\leq \mathbb{P}\big( \big| (\Delta t)^{-1} \int_{t_a}^{t_{a+1}} g(\tilde{\bgamma}_s) ds - \mathbb{E}^{P_*}[g] \big| \geq \delta / 2  \big) \\+
\mathbb{P}\big( (\Delta t)^{-1}\int_{t_a}^{t_{a+1}} \big| g(\bgamma_s) -g(\tilde{\bgamma}_s)  \big| ds \geq \delta / 2 \text{ and } \hat{\tau} \geq t_{a+1}  \big).
\end{multline*}
Since $\bupsilon_{\epsilon^2 (t-t_a)}$ has the same probability law as $\tilde{\bgamma}_{t-t_a}$, it follows that for any $\tilde{\epsilon} > 0$,
\begin{multline*}
\mathbb{P}\big( \big| (\Delta t)^{-1} \int_{t_a}^{t_{a+1}} g(\tilde{\bgamma}_s) ds - \mathbb{E}^{P_*}[g] \big| \geq \delta / 2  \big)
= \mathbb{P}\big( \big| T^{-1} \int_{0}^{T} g(\bupsilon^{\bgamma_a}_s) ds - \mathbb{E}^{P_*}[g] \big| \geq \delta / 2  \big) \leq \tilde{\epsilon},
\end{multline*}
thanks to Lemma \ref{Lemma T constant}, as long as $\epsilon$ is small enough.

Furthermore since $g$ has Lipschitz constant upperbounded by $m$ (noted in \eqref{eq: g lipschitz}),
\begin{multline*}
\mathbb{P}\big( (\Delta t)^{-1}\int_{t_a}^{t_{a+1}} \big| g(\bgamma_s) -g(\tilde{\bgamma}_s)  \big| ds \geq \delta / 2 \text{ and } \hat{\tau} \geq t_{a+1}  \big) \\ \leq \mathbb{P}\big( \sup_{t\in [t_a, t_{a+1}]} \|\bgamma_t -  \tilde{\bgamma}_t\| \geq \delta / (2m) \text{ and } \hat{\tau} \geq t_{a+1}  \big)
\leq \tilde{\epsilon},
\end{multline*}
using the result of Lemma \ref{Lemma beta tilde beta}. Let $\hat{\mathfrak{a}} = \lfloor (\Delta t)^{-1}T_{\delta} \rfloor - 1$.
The previous two results imply that
\begin{align}
\mathbb{E}\big[ \chi \lbrace \hat{\tau} \geq T_{\delta}\rbrace &\exp\big( -\kappa \hat{\tau} \delta \big) \prod_{a=0}^{\mathfrak{a}} H_a \big] \leq  \big[ 2 \tilde{\epsilon} \exp ( \kappa \Delta t ) + (1-2\tilde{\epsilon})\exp(\kappa \delta) \big]^{\hat{\mathfrak{a}}}\exp\big(-\kappa \delta T_{\delta} \big)\nonumber \\
&= \exp( \hat{\mathfrak{a}}\kappa \delta-\kappa \delta T_{\delta}  ) \big\lbrace 1 + 2\tilde{\epsilon}\big( \exp(\kappa \Delta t - \kappa \delta ) - 1 \big) \big\rbrace^{\hat{\mathfrak{a}}}\nonumber \\
&\leq \exp\big\lbrace \hat{\mathfrak{a}} \kappa \delta + 2\tilde{\epsilon}\hat{\mathfrak{a}} \big( \exp(\kappa \Delta t - \kappa \delta ) - 1 \big) -\kappa \delta  T_{\delta}\big\rbrace .
\end{align}
We substitute $\kappa = \epsilon^{2}$. We take $T$ to be large enough that
\begin{equation}
 \hat{\mathfrak{a}} \leq T_{\delta} / 3,
\end{equation}
and we take $\tilde{\epsilon}$ to be small enough that 
\begin{equation}
 2\tilde{\epsilon} \big( \exp( T - \epsilon^2 \delta ) - 1 \big) \leq \delta \epsilon^{-2}  T_{\delta} / 3,
\end{equation}
which is always possible, since $T_{\delta} = \exp( C_{\delta}\epsilon^{-2})$. In this way we find that
\[
\exp\big\lbrace \hat{\mathfrak{a}} \kappa \delta + 2\tilde{\epsilon}\hat{\mathfrak{a}} \big( \exp(\kappa \Delta t - \kappa \delta ) - 1 \big) -\kappa \delta  T_{\delta}\big\rbrace \leq \exp\big\lbrace - \epsilon^{-2}\delta \exp(-C_{\delta}\epsilon^{-2}) / 3\big\rbrace \\
\leq \frac{1}{2} \exp\big( -C_{\delta}\epsilon^{-2} \big),
\]
for small enough $\epsilon$. Substituting this bound into \eqref{eq: end chebysehev}, we have proved \eqref{eq: to establish C delta epsilon minus 2}, i.e.
\begin{equation}
\mathbb{P}\big(\big|T_{\delta}^{-1}\int_0^{T_{\delta}} g(\bgamma_s) ds-\mathbb{E}^{P_*}[g] \big| \geq \delta  \text{ and }\hat{\tau} \geq T_{\delta} \big) \leq \exp( -C_{\delta}\epsilon^{-2}) 
\end{equation}
We have thus established \eqref{eq: first result last theorem}.

We now prove \eqref{eq: second result last theorem}.
\begin{proof}
Notice that, substituting the SDE for $\bgamma_t$ in Lemma \ref{Lemma gamma SDE} and applying a union of events bound,
\begin{multline}
\mathbb{P}\big( \sup_{1\leq i \leq m}\big| \epsilon^{-2}  T_{\delta}^{-1} \gamma^i_{T_{\delta}}-\mathbb{E}^{P_*}[ \tilde{\mathcal{V}}_i] \big| > \delta \big)  \leq \mathbb{P}\big( \hat{\tau} \leq T_{\delta} \big)\\
+ \mathbb{P}\big(  \hat{\tau} > T_{\delta} \text{ and } \sup_{1\leq i \leq m}\big| 1/(2T_{\delta})\sum_{j=1}^\infty \int_0^{T_{\delta}} D^{(2)}\Theta_i(u_t)\cdot B(u_t)e_j \cdot B(u_t)e_j dt-\mathbb{E}^{P_*}[ \tilde{\mathcal{V}}_i] \big| > \delta / 2 \big) \\ + \mathbb{P}\big( \hat{\tau} > T_{\delta} \text{ and } \sup_{1\leq i \leq m}\big| \epsilon^{-1}  T_{\delta}^{-1} \int_0^{T_{\delta}}D\Theta_i(u_t)dW_t \big| > \delta / 2 \big). \label{eq: last equatino split}
\end{multline}
We have already seen that $\mathbb{P}\big( \hat{\tau} \leq T_{\delta} \big) \leq \exp(-C_{\delta}\epsilon^{-2})$. For the second term, employing a union-of-events bound, and Chernoff's Inequality, for a constant $\kappa > 0$,
\begin{multline}
 \mathbb{P}\big(\hat{\tau} > T_{\delta} \text{ and } \sup_{1\leq i \leq m}\big| \epsilon^{-1}  T_{\delta}^{-1} \int_0^{T_{\delta}}D\Theta_i(u_t)dW_t \big| > \delta / 2 \big)
 \leq \\ \sum_{1\leq i \leq m}\mathbb{P}\big( \hat{\tau} \geq T_{\delta} \text{ and }\epsilon^{-1}  T_{\delta}^{-1} \int_0^{T_{\delta}}D\Theta_i(u_t) dW_t  > \delta / 2 \big)+\\
 \sum_{1\leq i \leq m} \mathbb{P}\big(\hat{\tau} > T_{\delta} \text{ and } \epsilon^{-1}  T_{\delta}^{-1} \int_0^{T_{\delta}}D\Theta_i(u_t) dW_t  < -\delta / 2 \big)\\
 \leq \sum_{1\leq i \leq m}\mathbb{E}\big[\chi \lbrace \hat{\tau} > T_{\delta}\rbrace \exp\big\lbrace \kappa\int_0^{T_{\delta}}D\Theta_i(u_t)dW_t - \kappa \delta \epsilon T_{\delta} / 2 \big\rbrace +\\ \chi \lbrace \hat{\tau} > T_{\delta}\rbrace \exp\big\lbrace -\kappa\int_0^{T_{\delta}}D\Theta_i(u_t)dW_t - \kappa \delta \epsilon T_{\delta} / 2 \big\rbrace \big] .\label{eq: to show penultimate}
\end{multline}
Using the bound in Lemma \ref{Lemma Theta Derivative Approximation}, as long as $\zeta$ is sufficiently small, $|D\Theta_i(u_t)|$ is uniformly bounded as long as $t\leq \hat{\tau}$, and we thus find that
\begin{align*}
\frac{1}{2}  \int_0^{T_{\delta}} \big\lbrace D\Theta^i(u_t) \big\rbrace^2 dt  \leq \bar{C}T_{\delta},
\end{align*}
for a constant $\bar{C} > 0$. We thus find that
\begin{multline*}
\mathbb{E}\big[ \chi\big\lbrace \hat{\tau} > T_{\delta} \big\rbrace \exp\big\lbrace \kappa\int_0^{T_{\delta}}D\Theta^i(u_t) dW_t - \kappa \delta \epsilon T_{\delta} / 2 \big\rbrace \big]
\leq \mathbb{E}\big[ \chi\big\lbrace \hat{\tau} > T_{\delta} \big\rbrace\exp\big\lbrace \kappa\int_0^{T_{\delta}}D\Theta^i(u_t) dW_t \\  - \frac{ \kappa^2}{2} \int_0^{T_{\delta}}\lbrace D\Theta^i(u_t) \rbrace^2 dt + \kappa^2\bar{C}T_{\delta} - \kappa \delta \epsilon T_{\delta} / 2 \big\rbrace \big] 
= \exp\big\lbrace T_{\delta} \kappa^2 \bar{C} - \kappa \delta \epsilon T_{\delta} / 2 \big\rbrace,
\end{multline*}
by Girsanov's Theorem. We choose $\kappa = \delta \epsilon / (4\bar{C})$, and we obtain that 
\begin{equation}
\mathbb{E}\big[ \chi\big\lbrace \hat{\tau} > T_{\delta} \big\rbrace \exp\big\lbrace \kappa\int_0^{T_{\delta}}D\Theta^i(u_t) dW_t - \kappa \delta \epsilon T_{\delta} / 2 \big\rbrace \big] \leq \exp\big\lbrace -\delta^2 \epsilon^2 T_{\delta} / (8 \bar{C}) \big\rbrace .
\end{equation}
Since each of the terms in \eqref{eq: to show penultimate} can be bounded in the same manner as the above, we obtain that
\begin{equation}
 \mathbb{P}\big( \sup_{1\leq i \leq m}\big| \epsilon^{-1}  T_{\delta}^{-1} \int_0^{T_{\delta}}D\Theta^i(u_t) dW_t \big| > \delta / 2 \big) \leq 2m \exp\big\lbrace -\delta^2 \epsilon^2 T_{\delta} / (8 \bar{C}) \big\rbrace .
\end{equation}
Since $T_\delta = \exp( C_{\delta}\epsilon^{-2})$, in the limit as $\epsilon \to 0$, $\epsilon^2 T_{\delta} \gg \epsilon^{-2}$. For the remaining term in \eqref{eq: last equatino split}, using a union of events bound,
\begin{multline}\label{eq: temporary drift lipschitz 1}
\mathbb{P}\big(\hat{\tau} > T_{\delta} \text{ and }  \sup_{1\leq i \leq m}\big| 1/(2T_{\delta})\sum_{j=1}^\infty \int_0^{T_{\delta}} D^{(2)}\Theta_i(u_t)\cdot B(u_t)e_j \cdot B(u_t)e_j dt-\mathbb{E}^{P_*}[ \tilde{\mathcal{V}}_i] \big| > \delta / 2 \big) \leq \\ \mathbb{P}\big( \hat{\tau} > T_{\delta} \text{ and }\sup_{1\leq i \leq m}\big\lbrace T_{\delta}^{-1} \int_0^{T_{\delta}} \big|  \frac{1}{2}\sum_{j=1}^\infty D^{(2)}\Theta_i(u_t)\cdot B(u_t)e_j \cdot B(u_t)e_j- \tilde{\mathcal{V}}_i(\bgamma_t) \big| dt \big\rbrace > \delta / 4 \big) \\ +
\mathbb{P}\big(\hat{\tau} > T_{\delta} \text{ and }  \sup_{1\leq i \leq m}\big| T_{\delta}^{-1} \int_0^{T_{\delta}}\tilde{\mathcal{V}}_i(\bgamma_t )dt-\mathbb{E}^{P_*}[ \tilde{\mathcal{V}}_i] \big| > \delta / 4 \big)
\end{multline}
Now for small enough $\zeta$, thanks to Lemma \ref{Lemma Theta Derivative Approximation}, whenever $t \leq \hat{\tau}$
\[
 \big|  \frac{1}{2}\sum_{j=1}^\infty D^{(2)}\Theta_i(u_t)\cdot B(u_t)e_j \cdot B(u_t)e_j- \tilde{\mathcal{V}}_i(\bgamma_t) \big| \leq \delta / 4 ,
\]
and the first term on the right hand side of \eqref{eq: temporary drift lipschitz 1} has probability zero.

Using the first result of the lemma, i.e. \eqref{eq: first result last theorem}, since $\tilde{\mathcal{V}}_i$ is smooth and bounded, 
\begin{align*}
\mathbb{P}\big(\hat{\tau} > T_{\delta} \text{ and }  \sup_{1\leq i \leq m}\big| T_{\delta}^{-1} \int_0^{T_{\delta}}\tilde{\mathcal{V}}_i(\bgamma_t )dt-\mathbb{E}^{P_*}[ \tilde{\mathcal{V}}_i] \big| > \delta / 4 \big) \leq \exp\big( -C_{\delta}X^{-2} \epsilon^{-2} \big),
\end{align*}
where
\[
X = \sup_{1\leq i,j,k \leq m, \balpha\in\mathbb{R}^m} \big\lbrace \big| \tilde{\mathcal{V}}_i(\balpha) \big|,\big| \partial / \partial \alpha^j \tilde{\mathcal{V}}_i (\balpha)\big| , \big| \partial^2 / \partial \alpha^j \partial \alpha^k \tilde{\mathcal{V}}_i (\balpha)\big| \big\rbrace .
\]
In summary, collecting all of the above terms, we have proved that for small enough $\epsilon > 0$, there is a constant such that
\begin{equation}
\mathbb{P}\big( \sup_{1\leq i \leq m}\big| \epsilon^{-2}  T_{\delta}^{-1} \beta^i_{T_{\delta}}-\mathbb{E}^{P_*}[ \tilde{\mathcal{V}}_i] \big| > \delta \big) \leq  \exp\big( -\text{Const} \times\epsilon^{-2} \big).
\end{equation}
\end{proof}

\begin{lemma}\label{Lemma beta tilde beta}
For any $\bar{\epsilon},\tilde{\epsilon} > 0$, there exists $\zeta > 0$ and $\epsilon_0 , C_{\delta} > 0$ such that for all $\epsilon < \epsilon_0$,
\begin{align}
\mathbb{P}\big( \sup_{t\in [t_a , t_{a+1}]} \| \bgamma_t - \tilde{\bgamma}_t\| \geq \bar{\epsilon} \text{ and }\sup_{t\in [t_a , t_{a+1}]} \norm{v_t} \leq \zeta  \big) &<\tilde{\epsilon} \label{eq: to establish last lemma }\\
\mathbb{P}\big( \hat{\tau} < T_{\delta}  \big) &\leq \frac{1}{2}\exp\big(-C_{\delta} \epsilon^{-2} \big) .
\end{align}
\end{lemma}
\begin{proof}
It follows from Lemma \ref{Lemma Theta Derivative Approximation} that
\begin{equation}
\| \frac{1}{2}\sum_{j=1}^\infty D^{(2)}\Theta_i(u_t) \cdot B(u_t)e_j \cdot  B(u_t)e_j  - \tilde{\mathcal{V}}_i(\bgamma_t) \| = O\big( \norm{u - \varphi_{\bgamma_t}} \big) \\
= O\big( \norm{u - \varphi_{\bbeta_t}} \big).
\end{equation}
The rest of the proof then follows using similar techniques to the rest of the paper, and is omitted.
\end{proof}



\bibliographystyle{siam}
\bibliography{StochasticWaves2}

\end{document}